\documentclass[11pt]{amsart}
\usepackage{amsmath, amssymb, comment, graphicx, url, amsthm, multicol, latexsym, enumerate, bbm, mathtools, physics, microtype, cite, xcolor, float}
\usepackage{units}
\usepackage[all,cmtip]{xy}
\usepackage{wrapfig}
\usepackage[hmargin=1in]{geometry} 
\usepackage{tikz}
\usepackage{tikzsymbols}
\usetikzlibrary{shapes.geometric, arrows}
\usepackage[cm]{sfmath}
\usepackage{hyperref}
\usepackage{indentfirst}
\newcommand{\sym}[1]{\mathfrak{S}_{#1}}
\usepackage{subcaption}

\definecolor{andresblue}{rgb}{0,0.72,0.92}
\definecolor{andrespink}{rgb}{1,0,1}
\definecolor{orange}{rgb}{1,0.5,0}

\tikzset{
    back/.style={
        loosely dotted,
        thin
    },
    edge/.style={
        color=black,
        thick
    },
    facet/.style={
        fill=andresblue,
        fill opacity=0.333333
    },
    vertex/.style={
        inner sep=1.5pt,
        circle,
        draw=black,
        fill=andrespink,
        thick
    },
    subvertex/.style={
        inner sep=1.5pt,
        circle,
        draw=orange!75!black,
        fill=orange,
        thick
    },
    subedge/.style={
        color=orange!95!black,
        thick
    },
    subfacet/.style={
        fill=orange!95!black,
        fill opacity=0.5
    },
    subback/.style={
        dotted,
        thick
    },
    slabel/.style n args={2}{
        label={[font=\scriptsize,black,#1]:{#2}}
    },
}

\hypersetup{
    colorlinks=true,
    linkcolor=blue,
    citecolor=magenta,
    filecolor=magenta,      
    urlcolor=magenta,
}

\newtheorem{theorem}{Theorem}[section]
\newtheorem{proposition}[theorem]{Proposition}

\newtheorem{conjecture}[theorem]{Conjecture}

\newtheorem{lemma}[theorem]{Lemma}
\newtheorem{corollary}[theorem]{Corollary}

\theoremstyle{definition}
\newtheorem{remark}[theorem]{Remark}
\newtheorem{definition}[theorem]{Definition}

\newtheorem{example}[theorem]{Example}

\theoremstyle{remark}

\newcommand{\defterm}[1]{\emph{#1}}

\newcommand{\R}{\mathbb{R}}

\newcommand{\Z}{\mathbb{Z}}

\newcommand{\Q}{\mathbb{Q}}

\renewcommand{\vb}{\mathbf{b}}

\newcommand{\ve}{\mathbf{e}}

\newcommand{\vv}{\mathbf{v}}
\newcommand{\vx}{\mathbf{x}}
\newcommand{\vz}{\mathbf{z}}

\newcommand{\etal}{\textit{et al.} }

\newcommand{\ehr}{\mathrm{ehr}}

\newcommand{\Lp}{\mathfrak{L}}

\newcommand{\X}{\mathfrak{X}}

\DeclareMathOperator{\conv}{conv}
\DeclareMathOperator{\aff}{aff}

\DeclareMathOperator{\vol}{vol}

\makeatletter % to repeat theorem numbers
\newtheorem*{rep@theorem}{\rep@title}\newcommand{\newreptheorem}[2]{%
\newenvironment{rep#1}[1]{%
\def\rep@title{\bf #2 \ref{##1}}%
\begin{rep@theorem}}%
{\end{rep@theorem}}}
\makeatother
\newreptheorem{theorem}{Theorem}

\usepackage[colorinlistoftodos]{todonotes}

\definecolor{munsell}{rgb}{0.0, 0.5, 0.69}

\begin{document}

\title[Slices, Ehrhart polynomials, and magic positivity]{Lattice slices, Ehrhart polynomials, and magic positivity \\ of generalized parking-function polytopes}

\author[Hill]{Charlie Hill}
\address{Department of Mathematics\\
Harvey Mudd College
}
\email{chhill@g.hmc.edu }

\author[Luo]{Ambrose Luo}
\address{Department of Mathematics\\
Harvey Mudd College
}
\email{amluo@g.hmc.edu }

\author[Trinh]{Vu Trinh}
\address{Department of Mathematical Sciences\\
Claremont McKenna College
}
\email{vtrinh28@students.claremontmckenna.edu}

\author[Vindas-Mel\'{e}ndez]{Andr\'{e}s R. Vindas-Mel\'{e}ndez}
\address{Department of Mathematics\\
Harvey Mudd College
}
\email{avindasmelendez@g.hmc.edu}
\urladdr{https://math.hmc.edu/arvm}

%%%%%%%%%%%%%%%%%%%%%%%%%%%%

\begin{abstract}
For $\mathbf{b}=(b_1,\dots,b_n)\in\mathbb{Z}_{>0}^n$, a $\mathbf{b}$-parking function is a sequence $(\beta_1,\dots,\beta_n)$ of positive integers whose nondecreasing rearrangement $\beta_1'\le\beta_2'\le\cdots\le\beta_n'$ satisfies $\beta_i'\le b_1+\cdots+b_i$.
The $\mathbf{b}$-parking-function polytope $\mathfrak{X}_n(\mathbf{b})$ is the convex hull of all $\mathbf{b}$-parking functions of length $n$ in $\mathbb{R}^n$.
We prove that every lattice slice of $\mathfrak{X}_n(\mathbf{b})$, obtained by fixing one coordinate at an integer value, is itself a $\mathbf{b}'$-parking-function polytope of one dimension less, with an explicit parameter vector $\mathbf{b}'$; this yields a recursion for the number of lattice points of $\mathfrak{X}_n(\mathbf{b})$.
We further show that every dilate of a $\mathbf{b}$-parking-function polytope is a translate of another such polytope, that the number of lattice points is a polynomial function of $\mathbf{b}$, and we deduce an explicit formula for the Ehrhart polynomial of $\mathfrak{X}_n(\mathbf{b})$ for arbitrary $\mathbf{b}$ as a finite sum indexed by draconian sequences, resolving a problem of Hanada, Lentfer, and Vindas-Mel\'endez; an equivalent formula was recently obtained, independently, by Liu and Thawinrak in a closely related setting.
In the special case $\mathbf{b}=(a,b,\dots,b)$, we obtain an explicit closed form and a generating function for the Ehrhart polynomial.
As an application, we classify magic positivity in the two-parameter family $\mathfrak{X}_n(a,b)=\mathfrak{X}_n(a,b,\dots,b)$: the polytope $\mathfrak{X}_n(a,b)$ is magic positive if and only if $(n,a,b)\ne(2,1,1)$.
Thus, we answer a problem posed by Ferroni and Higashitani for $\mathfrak{X}_n(a,b)$.
Our result extends recent work of Liu and Zhang on partial permutahedra and leads us to conjecture that magic positivity holds for every $\mathfrak{X}_n(\mathbf{b})$ with $n\ge3$.

\end{abstract}

%%%%%%%%%%%%%%%%%%%%%%%%%%%%

\maketitle

%%%%%%%%%%%%%%%%%%%%%%%%%%%%

\section{Introduction}\label{sec:intro}

A \emph{classical parking function} of length $n$ is a list $(\alpha_1,\alpha_2,\dots,\alpha_n)$ of positive integers whose nondecreasing rearrangement $\alpha_1'\le\alpha_2'\le\cdots\le\alpha_n'$ satisfies $\alpha_i'\le i$.
It is well known that the number of classical parking functions of length $n$ is $(n+1)^{n-1}$, a quantity that also counts, for example, labeled rooted forests on $n$ vertices and regions of the Shi arrangement in $\R^n$; see \cite{Yan} for further discussion.
Let $PF_n$ denote the convex hull in $\R^n$ of all classical parking functions of length $n$.
In 2020, Stanley \cite{Stanley12191} asked for the volume and the number of vertices, facets, and lattice points of $PF_n$.
These questions were answered by Amanbayeva and Wang \cite{AmanbayevaWang} and independently (in part) by Stong \cite{Stong}.
Subsequently, Behrend \cite{Behrend} obtained a general formula for the Ehrhart polynomials of partial permutahedra, a family introduced in \cite{BCCDEHI,HeuerStriker} that includes $PF_n$ up to integral equivalence.

In this paper we study the convex hull of a generalization of classical parking functions, namely $\vb$-parking functions, which have been explored from an enumerative perspective by Yan \cite{Yan} and Pitman and Stanley \cite{PitmanStanley}.

\begin{definition}\label{def:bpf}
Let $\vb=(b_1,\dots,b_n)\in\Z_{>0}^n$.
A \emph{$\vb$-parking function} is a list $(\beta_1,\dots,\beta_n)$ of positive integers whose nondecreasing rearrangement $\beta_1' \le \beta_2' \le \cdots \le \beta_n'$ satisfies $\beta_i' \le  b_1 + \cdots + b_i$.
\end{definition}

The \emph{$\vb$-parking-function polytope} $\X_n(\vb)$ is the convex hull of all $\vb$-parking functions of length $n$ in $\R^n$.
For positive integers $a$ and $b$, we use the abbreviation
\[
\X_n(a,b):=\X_n(a,b,\dots,b).
\]
This two-parameter family was introduced and studied by Hanada, Lentfer, and Vindas-Mel\'endez \cite{HanadaLentferVindas}, who investigated its face structure and computed its volume.
Bayer \etal \cite{BBCDDDLMMNV} subsequently studied the full family $\X_n(\vb)$ for arbitrary $\vb\in\Z_{>0}^n$: they gave minimal vertex and inequality descriptions, showed that a lifting of $\X_n(\vb)$ is a generalized permutahedron, identified its combinatorial type, and studied it from the perspectives of building sets and polymatroids.
Very recently, Liu and Thawinrak \cite{LiuThawinrak} introduced a broader family $PF(\mathbf u)$, indexed by arbitrary nondecreasing vectors $\mathbf u\in\R^n_{\ge0}$ (so that ties among the entries are allowed), and determined its normal fans, face posets, and $h$-polynomials; every $\X_n(\vb)$ is a lattice translate of a member of their family (see Remark~\ref{rem:LT}).
Hanada, Lentfer, and Vindas-Mel\'endez pose the problem of enumerating the lattice points of $\X_n(\vb)$ and, more generally, of determining its Ehrhart polynomial $\ehr_{\X_n(\vb)}(t):=\#\bigl(t\X_n(\vb)\cap\Z^n\bigr)$; see \cite[Problem~5.8]{HanadaLentferVindas}, where the case $\vb=(a,b,\dots,b)$ is singled out as open.
In this paper we solve these problems, and we apply the solution to establish a strong positivity property of the resulting Ehrhart polynomials.

Our starting point is a structural result about lattice slices.
By the symmetry of $\X_n(\vb)$ we may fix the last coordinate; for an integer $h$, the \emph{layer at height $h$} is the fiber $\X_n(\vb)[h]:=\{\vx\in\R^{n-1}:(\vx,h)\in\X_n(\vb)\}$.
We prove that every such layer is again a parking-function polytope and identify its parameter vector explicitly.

\begin{reptheorem}{thm:slice}
Let $n\ge 2$, and write $h=S_{\ell-1}+r$ with $1\le\ell\le n$ and $1\le r\le b_\ell$, where $S_i:=b_1+\cdots+b_i$.
Then $\X_n(\vb)[h]=\X_{n-1}(\vb')$, where $\vb'\in\Z_{>0}^{\,n-1}$ is given by
\[
\vb'\ =\
\begin{cases}
(\,b_1+b_2,\ b_3,\ \dots,\ b_n\,), & \ell=1,\\
(\,b_1,\dots,b_{\ell-2},\ b_{\ell-1}+b_\ell-r,\ b_{\ell+1}+r,\ b_{\ell+2},\dots,b_n\,), & 1<\ell<n,\\
(\,b_1,\dots,b_{n-2},\ b_{n-1}+b_n-r\,), & \ell=n.
\end{cases}
\]
\end{reptheorem}

Summing over the layers yields a recursion (Theorem~\ref{thm:recursion}) for the lattice-point count $L(\vb):=\#\bigl(\X_n(\vb)\cap\Z^n\bigr)$, which determines $L(\vb)$ for every $\vb$ from the initial condition $L(b_1)=b_1$.

We then turn to Ehrhart theory.
Two further structural facts allow us to leverage the recursion.
First, every dilate of a $\vb$-parking-function polytope is a lattice translate of another $\vb$-parking-function polytope (Lemma~\ref{lem:dilation}).
Second, the lattice-point count $L(\vb)$, initially defined only as a recursively computed counting function on positive integer parameter vectors, is in fact the restriction of a single polynomial in the parameters $b_1,\dots,b_n$ (Proposition~\ref{prop:polynomiality}).
Combining these facts with a lattice-point formula of Postnikov \cite{Postnikov} for trimmed generalized permutahedra, we obtain an explicit formula for the Ehrhart polynomial of $\X_n(\vb)$ for \emph{arbitrary} $\vb$, as a finite sum indexed by draconian sequences, that is, by multisets of subsets of $[n]$ admitting a system of distinct representatives (Theorem~\ref{thm:general}).

Notably, the formula holds for all $\vb$ even though the underlying Minkowski-sum decomposition of $\X_n(\vb)$ has negative coefficients in general: both sides of the formula are polynomial functions of $\vb$, so their agreement on the parameter vectors with nonnegative coefficients, which are plentiful enough to determine a polynomial, forces agreement everywhere.
An equivalent draconian-sum formula was obtained very recently by Liu and Thawinrak \cite[Corollary~7.6]{LiuThawinrak} for their family $PF(\mathbf u)$.
Their derivation invokes Postnikov's formula directly, which is stated for nonnegative Minkowski coefficients; the extension to mixed-sign coefficients is also available from the signed generalized-permutahedron theory of Jochemko and Ravichandran \cite{JochemkoRavichandran}, and our polynomiality argument gives an independent, family-specific route to it (see Remark~\ref{rem:LT}).

Specializing to $\vb=(a,b,\dots,b)$, the draconian sum collapses to a sum over graphs in which every connected component contains at most one cycle, and the machinery of exponential generating functions, adapted from Behrend's work on partial permutahedra \cite{Behrend}, produces a fully explicit answer, resolving the problem posed in \cite[Section~5.2]{HanadaLentferVindas}.

\begin{reptheorem}{thm:ehrhart-ab}
For all positive integers $a,b,n$,
\[
\begin{split}
\ehr_{\X_n(a,b)}(t)
=\frac{1}{2^{\,n}}\sum_{i=0}^{\lfloor n/2\rfloor}\sum_{j=2i}^{n}
(-1)^{i+1}&\binom{n}{\,n-j,\ j-2i,\ i,\ i\,}\,i!\,(2j-4i-3)!!\\[-1pt]
&\times(bt)^{\,j-i}\bigl((2bn-b+2a-2)t+2\bigr)^{\,n-j},
\end{split}
\]
where $(2j-4i-3)!!:=-\prod_{k=1}^{\,j-2i}(2k-3)$.
\end{reptheorem}

Finally, we study the positivity properties of these Ehrhart polynomials.
Recall that a lattice polytope $P$ of dimension $d$ is \emph{Ehrhart positive} if all coefficients of $\ehr_P(t)$ are nonnegative, and \emph{magic positive} if $\ehr_P(t)$ has nonnegative coefficients in the binomial-type basis $\{t^i(t+1)^{d-i}\}_{i=0}^d$; magic positivity implies Ehrhart positivity and, by a result of Br\"and\'en \cite{Branden}, real-rootedness of the $h^*$-polynomial.
Ferroni and Higashitani \cite[Problem~4.23]{FerroniHigashitani} include generalized parking-function polytopes among the families for which they ask whether, or to what extent, magic positivity holds.
Here we give a complete answer for the two-parameter family $\X_n(a,b)=\X_n(a,b,\dots,b)$ studied by Hanada, Lentfer, and Vindas-Mel\'endez \cite{HanadaLentferVindas}.

Liu and Zhang \cite{LiuZhang} recently classified magic positivity for partial permutahedra in the stable range: $\mathcal P(m,p)$ is magic positive for $p\ge m-1$, with the single exception of $\mathcal P(2,1)$.
Here and throughout, $\mathcal P(m,p)$ denotes the partial permutahedron on $m$ coordinates with second parameter $p$, denoted $\mathcal P(m,n)$ in \cite{BCCDEHI,Behrend,LiuZhang}; we reserve the letter $n$ for the length of the parameter vector $\vb$.
Since $\X_n(a,1)\cong\mathcal P(n,n+a-2)$ by \cite[Proposition~3.16]{HanadaLentferVindas}, our family extends these stable-range partial permutahedra by the additional parameter $b$.
We give the complete magic-positivity classification for this two-parameter family.

\begin{reptheorem}{thm:magic}
Let $a,b,n$ be positive integers. The polytope $\X_n(a,b)$ is magic positive if and only if $(n,a,b)\ne(2,1,1)$.
In the exceptional case, $\X_2(1,1)=PF_2$ and the coefficient of $t(t+1)$ is $-\tfrac12$.
\end{reptheorem}

Thus, Theorem~\ref{thm:magic} gives a complete answer to \cite[Problem~4.23]{FerroniHigashitani} for the two-parameter family $\X_n(a,b,\dots,b)$: every polytope in this family is magic positive except $\X_2(1,1)=PF_2$.
The corresponding question for arbitrary $\X_n(\vb)$ remains open; see Conjecture~\ref{conj:magicgeneral}.
Since the $h^*$-polynomial of $PF_2$ is the constant $1$, it follows that the $h^*$-polynomial of $\X_n(a,b)$ is real-rooted for \emph{all} positive integers $a,b,n$ (Corollary~\ref{cor:realrooted}).
Our proof follows the strategy of Liu and Zhang \cite{LiuZhang}, with two new ingredients: the closed form of Theorem~\ref{thm:ehrhart-ab} in the coefficient-extraction format, and refined coefficient estimates for a $b$-deformed version of their auxiliary series.
Computational evidence suggests that the phenomenon is not special to $\vb=(a,b,\dots,b)$, and we conjecture that $\X_n(\vb)$ is magic positive for every $\vb\in\Z_{>0}^n$ with $n\ge3$ (Conjecture~\ref{conj:magicgeneral}).

The paper is structured as follows:
\begin{itemize}
    \item In Section~\ref{sec:prelim} we recall the inequality description of $\X_n(\vb)$ from \cite{BBCDDDLMMNV}, fix notation, and review the necessary background from Ehrhart theory.
    \item In Section~\ref{sec:slices} we prove Theorem~\ref{thm:slice}, describing all lattice slices of $\X_n(\vb)$, and we describe the block structure of the layers.
    \item In Section~\ref{sec:recursion} we derive the recursion for the lattice-point count $L(\vb)$ (Theorem~\ref{thm:recursion}).
    \item In Section~\ref{sec:general} we prove the dilation lemma, the polynomiality of $L(\vb)$, and an explicit formula for $\ehr_{\X_n(\vb)}(t)$ for arbitrary $\vb$ (Theorem~\ref{thm:general}; compare with \cite{LiuThawinrak}), together with a volume formula (Corollary~\ref{cor:volume}).
    \item In Section~\ref{sec:ehrhart} we specialize to $\vb=(a,b,\dots,b)$ and prove Theorem~\ref{thm:ehrhart-ab}, as well as the generating-function identity \eqref{eq:gf-ab}.
    \item In Section~\ref{sec:magic} we prove Theorem~\ref{thm:magic} and Corollary~\ref{cor:realrooted}.
    \item We conclude in Section~\ref{sec:future} with some directions for further study.
\end{itemize}

%%%%%%%%%%%%%%%%%%%%%%%%%%%%

\section{Preliminaries}\label{sec:prelim}

\subsection{The \texorpdfstring{$\vb$}{b}-parking-function polytope}\label{subsec:polytope}

Throughout, fix a positive integer $n$ and a vector $\vb=(b_1,\dots,b_n)\in\Z_{>0}^n$, and write $[n]:=\{1,2,\dots,n\}$.
Following~\cite{BBCDDDLMMNV}, set
\[
S_i:=\sum_{j=1}^{i}b_j\quad(0\le i\le n,\ \text{with } S_0:=0),
\qquad
\vv_k:=(1,\dots,1,S_{k+1},S_{k+2},\dots,S_n)\in\R^n\quad(0\le k\le n).
\]

\begin{definition}[{\cite[Definition 2.1]{BBCDDDLMMNV}}]\label{def:poly}
The \emph{$\vb$-parking-function polytope} $\X_n(\vb)$ is the convex hull of all $\vb$-parking functions of length $n$ in $\R^n$.
\end{definition}

The polytope $\X_n(\vb)$ is invariant under the action of the symmetric group $\sym{n}$ permuting coordinates, and its vertices are the points $\pi(\vv_k)$ for $\pi\in\sym{n}$ and $0\le k\le n$~\cite[Theorem 2.3(b)]{BBCDDDLMMNV}. 
We use the following inequality description.

\begin{theorem}[{\cite[Theorem 2.3(c)]{BBCDDDLMMNV}}]\label{thm:ineq}
The polytope $\X_n(\vb)$ is the set of points $(x_1,\dots,x_n)\in\R^n$ such that $x_i\ge 1$ for $1\le i\le n$ and
\begin{equation}\tag{$\vb$-parking}\label{eq:bpf}
\sum_{i\in I}x_i\ \le\ \sum_{i=n-|I|+1}^{n}S_i
\end{equation}
for all nonempty $I\subseteq[n]$.
These are the facet-defining inequalities for $\X_n(\vb)$ when $b_1\ge 2$; when $b_1=1$ the inequalities with $|I|=n-1$ are redundant.
In either case, every point of $\X_n(\vb)$ satisfies (\ref{eq:bpf}) for all nonempty $I\subseteq[n]$.
\end{theorem}

The right-hand side of (\ref{eq:bpf}) depends only on $|I|$. 
For convenience we abbreviate it by
\begin{equation}\label{eq:Sigma}
\Sigma_k\ :=\ \sum_{i=n-k+1}^{n}S_i\qquad(0\le k\le n,\ \text{with } \Sigma_0:=0),
\end{equation}
so that (\ref{eq:bpf}) reads $\sum_{i\in I}x_i\le\Sigma_{|I|}$. 
Since $\Sigma_k-\Sigma_{k-1}=S_{n-k+1}$, the sequence $\Sigma_0<\Sigma_1<\cdots<\Sigma_n$ is determined by $\vb$, and conversely $\vb$ is recovered from it. 
Throughout, we use freely that every point of $\X_n(\vb)$ satisfies (\ref{eq:bpf}) for all nonempty $I\subseteq[n]$ (Theorem~\ref{thm:ineq}), which allows us to describe $\X_n(\vb)$ by the full, possibly redundant, system.

\begin{example}\label{ex:running}
Throughout the paper we illustrate our results with the polytope $\X_3(3,2,2)$.
Here $(S_1,S_2,S_3)=(3,5,7)$, so $(\Sigma_1,\Sigma_2,\Sigma_3)=(7,12,15)$, and Theorem~\ref{thm:ineq} describes $\X_3(3,2,2)$ as the set of points $(x_1,x_2,x_3)\in\R^3$ satisfying
\[
x_i\ge1,\qquad
x_i\le 7,\qquad
x_i+x_j\le 12\ \ (i\ne j),\qquad
x_1+x_2+x_3\le 15.
\]
Its vertices are the permutations of $\vv_0=(3,5,7)$, $\vv_1=(1,5,7)$, $\vv_2=(1,1,7)$, and $\vv_3=(1,1,1)$; the polytope is shown in Figure~\ref{fig:running}.
Since $(3,2,2)$ has the form $(a,b,b)$ with $(a,b)=(3,2)$, this polytope also belongs to the family studied in Sections~\ref{sec:ehrhart} and~\ref{sec:magic}.
\end{example}

\begin{figure}[ht]
    \centering
    \begin{tikzpicture}[scale=0.72]
\draw[back, thick] (0.580,0.680) -- (0.580,6.680);
\draw[back, thick] (0.580,0.680) -- (6.580,0.680);
\draw[back, thick] (0.580,0.680) -- (-1.940,-1.240);
\fill[facet] (-1.940,-1.240) -- (2.060,-1.240) -- (2.060,0.760) -- (0.060,2.760) -- (-1.940,2.760) -- cycle;
\fill[facet] (4.900,1.400) -- (4.900,-0.600) -- (6.580,0.680) -- (6.580,4.680) -- (5.740,4.040) -- cycle;
\fill[facet] (-1.100,5.400) -- (0.900,5.400) -- (3.740,6.040) -- (4.580,6.680) -- (0.580,6.680) -- cycle;
\fill[facet] (2.060,-1.240) -- (4.900,-0.600) -- (4.900,1.400) -- (2.060,0.760) -- cycle;
\fill[facet] (-1.940,2.760) -- (0.060,2.760) -- (0.900,5.400) -- (-1.100,5.400) -- cycle;
\fill[facet] (5.740,4.040) -- (6.580,4.680) -- (4.580,6.680) -- (3.740,6.040) -- cycle;
\fill[facet] (0.060,2.760) -- (2.060,0.760) -- (4.900,1.400) -- (5.740,4.040) -- (3.740,6.040) -- (0.900,5.400) -- cycle;
\draw[edge] (0.580,6.680) -- (4.580,6.680);
\draw[edge] (0.580,6.680) -- (-1.100,5.400);
\draw[edge] (4.580,6.680) -- (6.580,4.680);
\draw[edge] (4.580,6.680) -- (3.740,6.040);
\draw[edge] (6.580,0.680) -- (6.580,4.680);
\draw[edge] (6.580,0.680) -- (4.900,-0.600);
\draw[edge] (6.580,4.680) -- (5.740,4.040);
\draw[edge] (3.740,6.040) -- (5.740,4.040);
\draw[edge] (3.740,6.040) -- (0.900,5.400);
\draw[edge] (5.740,4.040) -- (4.900,1.400);
\draw[edge] (-1.100,5.400) -- (0.900,5.400);
\draw[edge] (-1.100,5.400) -- (-1.940,2.760);
\draw[edge] (0.900,5.400) -- (0.060,2.760);
\draw[edge] (4.900,-0.600) -- (4.900,1.400);
\draw[edge] (4.900,-0.600) -- (2.060,-1.240);
\draw[edge] (4.900,1.400) -- (2.060,0.760);
\draw[edge] (-1.940,-1.240) -- (-1.940,2.760);
\draw[edge] (-1.940,-1.240) -- (2.060,-1.240);
\draw[edge] (-1.940,2.760) -- (0.060,2.760);
\draw[edge] (0.060,2.760) -- (2.060,0.760);
\draw[edge] (2.060,-1.240) -- (2.060,0.760);
\node[vertex] at (0.580,0.680) {};
\node[vertex] at (0.580,6.680) {};
\node[vertex] at (4.580,6.680) {};
\node[vertex] at (6.580,0.680) {};
\node[vertex] at (6.580,4.680) {};
\node[vertex] at (3.740,6.040) {};
\node[vertex] at (5.740,4.040) {};
\node[vertex] at (-1.100,5.400) {};
\node[vertex] at (0.900,5.400) {};
\node[vertex] at (4.900,-0.600) {};
\node[vertex] at (4.900,1.400) {};
\node[vertex] at (-1.940,-1.240) {};
\node[vertex] at (-1.940,2.760) {};
\node[vertex] at (0.060,2.760) {};
\node[vertex] at (2.060,-1.240) {};
\node[vertex] at (2.060,0.760) {};
\node[font=\scriptsize,below left=-1.5pt] at (0.580,0.680) {$(1,1,1)$};
\node[font=\scriptsize,above left=-1.5pt] at (0.580,6.680) {$(1,1,7)$};
\node[font=\scriptsize,above right=-1.5pt] at (4.580,6.680) {$(1,5,7)$};
\node[font=\scriptsize,right=-1.5pt] at (3.740,6.040) {$(3,5,7)$};
\end{tikzpicture}
    \caption{The running example $\X_3(3,2,2)$, with one vertex of each of the four types of Example~\ref{ex:running} labeled.}
    \label{fig:running}
\end{figure}
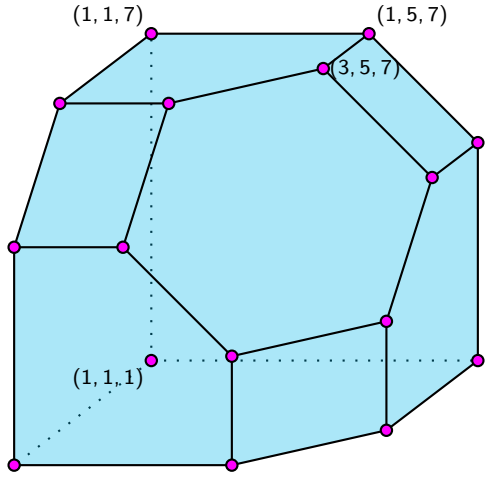

\subsection{Ehrhart theory}\label{subsec:ehrhart}

The \defterm{Ehrhart function} of a polytope $P\subseteq\R^n$ is
\[
\ehr_P(t)\ :=\ \#\bigl(tP\cap\Z^n\bigr),
\]
where $tP:=\{t\vx:\vx\in P\}$ denotes the $t$-th dilate of $P$ for a positive integer $t$.
A foundational theorem of Ehrhart \cite{Ehrhart} states that if $P$ is a \defterm{lattice polytope}, that is, if every vertex of $P$ has integer coordinates, then the Ehrhart function agrees with a polynomial in $t$ of degree $\dim P$, called the \defterm{Ehrhart polynomial} of $P$; its leading coefficient is the \defterm{relative volume} of $P$ (the volume taken with respect to the lattice $\Z^n\cap\aff(P)$), and its constant term is $1$.
We take care to maintain the distinction between the Ehrhart function and the Ehrhart polynomial: the enumerative identities in this paper are first established for the Ehrhart function, that is, at positive integer dilations, and Ehrhart's theorem then upgrades them to identities of polynomials.
Since the polytopes $\X_n(\vb)$ are full-dimensional (apart from the degenerate case $n=1=b_1$; see \cite[Proposition~2.2]{BBCDDDLMMNV}), for them the relative volume coincides with the Euclidean volume.

The Ehrhart polynomial of a lattice polytope $P$ of dimension $d$ can be written uniquely in the form
\[
\ehr_P(t)\ =\ \sum_{i=0}^{d}h_i^*\binom{t+d-i}{d},
\]
and the polynomial $h^*_P(z):=h_0^*+h_1^*z+\cdots+h_d^*z^d$ is called the \defterm{$h^*$-polynomial} of $P$; a theorem of Stanley \cite{StanleyDecompositions} asserts that its coefficients are nonnegative integers.
A lattice polytope $P$ is \defterm{Ehrhart positive} if all coefficients of its Ehrhart polynomial are nonnegative.
For a comprehensive treatment of Ehrhart theory we refer to the book of Beck and Robins \cite{BeckRobins}, and for a survey of positivity phenomena for Ehrhart coefficients and $h^*$-polynomials to Ferroni and Higashitani \cite{FerroniHigashitani}.

%%%%%%%%%%%%%%%%%%%%%%%%%%%%
\section{Lattice slices}\label{sec:slices}

By $\sym{n}$-invariance, fixing any single coordinate of $\X_n(\vb)$ yields the same family of slices; we fix the last coordinate $x_n$.

\begin{proposition}\label{prop:layers}
The values of $x_n$ on lattice points of $\X_n(\vb)$ are exactly $1,2,\dots,S_n$.
In particular, $\X_n(\vb)$ has $S_n=b_1+\cdots+b_n$ lattice layers in the last coordinate.
\end{proposition}

\begin{proof}
    Saying that integer values of $x_n$ are \emph{exactly} $1,\dots, S_n$ is a set equality, which we prove by two inclusions. 
    
    First, we show that if a lattice point has last coordinate $h$, then $1\leq h \leq S_n$, so no height outside this range can occur.
    Indeed, for every $\mathbf{x}\in\X_n(\vb)$, Theorem~\ref{thm:ineq} gives $x_n\ge 1$ directly, while taking $I=\{n\}$ in (\ref{eq:bpf}) gives $x_n\le\Sigma_1=S_n$.

 Then, we show that for every integer $h$ with $1\leq h \leq S_n$, there is a lattice point with last coordinate $h$, so every height in the range is attained with no gaps. 
 Note that the point $(1,\dots, 1, h)$ is a $\vb$-parking function: its nondecreasing rearrangement is the point itself and satisfies
 \[1\leq S_i \qquad \text{ for } 1\leq i \leq n-1, \qquad h\leq S_n.\]
Therefore, $(1,\dots, 1, h) \in \X_n(\vb) \cap \Z^n$, so the height $h$ is attained.
\end{proof}

For $1\le h\le S_n$, the \emph{layer at height $h$} is the fiber
\[
\X_n(\vb)[h]\ :=\ \bigl\{(x_1,\dots,x_{n-1})\in\R^{n-1}\ :\ (x_1,\dots,x_{n-1},h)\in\X_n(\vb)\bigr\}.
\]

We now come to the main result of this section (Theorem \ref{thm:slice}), which shows that each such layer is a polytope, a lower-dimensional $\vb$-parking-function polytope.
Informally, as the height $h$ increases from $1$ to $S_n$, the parameter vector of the layer sweeps through a sequence of $(n-1)$-tuples. 
It begins by merging the first two entries of $\vb$ into a single entry $b_1+b_2$, and thereafter, within each block of consecutive heights, it transfers weight one unit at a time from one coordinate to the next.

We prove this in three steps: we first record the inequalities that cut out a layer (Lemma \ref{lem:layerineq}), then determine which of the two competing bounds is active (Lemma \ref{lem:branch}), and finally read off the parameter vector of the layer coordinate by coordinate (Lemma \ref{lem:convert}).

\begin{lemma}\label{lem:layerineq}
    Fix $1 \le h \le S_n$. 
    The layer $\X_n(\vb)[h]$ is the set of $(x_1,\dots, x_{n-1})\in \R^{n-1}$ with $x_i \ge 1$ for $i \in [n-1]$ and 
    \begin{equation}\label{eq:layer}
        \sum_{i\in I}x_i \le \min\{\Sigma_k,\ \Sigma_{k+1}-h\}\qquad \text{for all nonempty } I \subseteq [n-1],\ |I|=k. 
    \end{equation}
\end{lemma}

\begin{proof}
    Set $x_n=h$ in the system of Theorem~\ref{thm:ineq}.
    For each nonempty $I\subseteq [n-1]$ with $|I|=k$, the (\ref{eq:bpf}) inequalities indexed by $I$ and by $I\cup\{n\}$ become
    \[
    \sum_{i\in I} x_i\le \Sigma_k
    \qquad\text{and}\qquad
    \sum_{i\in I} x_i \le \Sigma_{k+1}-h,
    \]
    respectively.
    As $I$ ranges over all nonempty subsets of $[n-1]$, these account for every inequality of the restricted system except the one indexed by $\{n\}$, which reads $h\le \Sigma_1=S_n$ and holds automatically by the hypothesis on $h$.
    Combining the two bounds for each $I$ gives the desired description.
\end{proof}

\begin{lemma}\label{lem:branch}
    Write $h=S_{\ell-1}+r$ with $1 \le \ell \le n$ and $1 \le r \le b_\ell$. 
    Then for $1\le k \le n-1$, 
    \[
    \min\{\Sigma_k,\ \Sigma_{k+1}-h\}=
    \begin{cases}
    \Sigma_k,&k\le n-\ell,\\
    \Sigma_{k+1}-h,&k\ge n-\ell+1.
    \end{cases}
    \]
\end{lemma}

\begin{proof}
    Since $\Sigma_{k+1} - \Sigma_k = S_{n-k}$, we have $\Sigma_k-(\Sigma_{k+1}-h)=h-S_{n-k}$, so $\Sigma_k\le\Sigma_{k+1}-h$ if and only if $h\le S_{n-k}$.
    As $S_{\ell-1}<h\le S_\ell$ and the partial sums $S_i$ are increasing, $h\le S_{n-k}$ holds exactly when $n-k\ge\ell$, that is, when $k\le n-\ell$.    
\end{proof}

\begin{lemma}\label{lem:convert}
    Set $B_k:=\min\{\Sigma_k,\Sigma_{k+1}-h\}$ for $1\le k\le n-1$ and $B_0:=0$, and define
\[
S_k'\ :=\ B_{n-k}-B_{n-k-1}\quad(1\le k\le n-1),
\qquad
b_k'\ :=\ S_k'-S_{k-1}'\quad(\text{with } S_0':=0).
\]
If $b_k'$ is a positive integer for every $1\le k\le n-1$, then $\X_n(\vb)[h]=\X_{n-1}(\vb')$ for $\vb'=(b_1',\dots,b_{n-1}')$.
\end{lemma}

\begin{proof}
    Under the hypothesis, the numbers $S_k'$ are the partial sums of $\vb'\in\Z_{>0}^{n-1}$, and the cardinality-$k$ bound of $\X_{n-1}(\vb')$ is 
    \[\Sigma_k'=\sum_{j=(n-1)-k+1}^{n-1}S_j',\]
    so that
    \[
    \Sigma_k'-\Sigma_{k-1}'\ =\ S_{n-k}'\ =\ B_k-B_{k-1}
    \qquad(1\le k\le n-1).
    \]
    Since $\Sigma_0'=0=B_0$, telescoping gives $\Sigma_k'=B_k$ for all $0\le k\le n-1$.
    Hence, the system of Lemma~\ref{lem:layerineq} cutting out $\X_n(\vb)[h]$ is exactly the full system $x_i\ge1$, 
    \[\sum_{i\in I}x_i\le\Sigma_{|I|}',\]
    which cuts out $\X_{n-1}(\vb')$ by Theorem~\ref{thm:ineq}.
\end{proof}

\begin{theorem}\label{thm:slice}
Let $n\ge 2$, and write $h=S_{\ell-1}+r$ with $1\le\ell\le n$ and $1\le r\le b_\ell$; equivalently, $S_{\ell-1}<h\le S_\ell$.
Then $\X_n(\vb)[h]=\X_{n-1}(\vb')$, where $\vb'\in\Z_{>0}^{\,n-1}$ is given by
\[
\vb'\ =\
\begin{cases}
(\,b_1+b_2,\ b_3,\ \dots,\ b_n\,), & \ell=1,\\
(\,b_1,\dots,b_{\ell-2},\ b_{\ell-1}+b_\ell-r,\ b_{\ell+1}+r,\ b_{\ell+2},\dots,b_n\,), & 1<\ell<n,\\
(\,b_1,\dots,b_{n-2},\ b_{n-1}+b_n-r\,), & \ell=n.
\end{cases}
\]
\end{theorem}

\begin{proof}
By Lemma \ref{lem:layerineq}, $\X_n(\vb)[h]$ is cut out by $x_i\ge 1$ together with the bounds $\sum_{i \in I}x_i \le B_k$, where $B_k=\min\{\Sigma_k, \Sigma_{k+1}-h\}$ and $k=|I|$.
By Lemma \ref{lem:convert}, it therefore suffices to compute the numbers $S_k'=B_{n-k}-B_{n-k-1}$ and the successive differences $b_k'=S_k'-S_{k-1}'$ (with $S_0':=0$), and to verify that every $b_k'$ is a positive integer.

By Lemma \ref{lem:branch}, $B_m=\Sigma_m$ when $m\le n-\ell$ and $B_m=\Sigma_{m+1}-h$ when $m \ge n-\ell +1$. 
Thus, the two terms $B_{n-k}$ and $B_{n-k-1}$ forming $S_k'$ fall into the same piece except for a single value of $k$, and we treat the three resulting cases in turn. 
Throughout we use $\Sigma_p-\Sigma_{p-1}=S_{n-p+1}$. 

\bigskip 
\noindent \emph{Case 1: $1\le k \le \ell-2$ (both terms in the $\Sigma_\square-h$ piece).}
Here $n-k-1\ge n-\ell+1$, so $B_{n-k}=\Sigma_{n-k+1}-h$ and $B_{n-k-1}=\Sigma_{n-k}-h$, hence,
\[
S_k'=(\Sigma_{n-k+1}-h)-(\Sigma_{n-k}-h)=\Sigma_{n-k+1}-\Sigma_{n-k}=S_k.
\]
Therefore, $b_k'=S_k'-S_{k-1}'=S_k-S_{k-1}=b_k$ (and $b_1'=S_1'-S_0'=b_1$).

\bigskip
\noindent\emph{Case 2: $k=\ell-1$ (each term on a different piece).} Now $B_{n-\ell+1}=\Sigma_{n-\ell+2}-h$ while $B_{n-\ell}=\Sigma_{n-\ell}$, so,
\[
S_{\ell-1}'=(\Sigma_{n-\ell+2}-h)-\Sigma_{n-\ell}=(\Sigma_{n-\ell+2}-\Sigma_{n-\ell})-h=S_{\ell-1}+S_\ell-h.
\]
Since $S_{\ell-2}'=S_{\ell-2}$ by Case 1, and using $S_{\ell-1}-S_{\ell-2}=b_{\ell-1}$ together with $S_\ell-h=b_\ell-r$ (as $h=S_{\ell-1}+r$),
\[
b_{\ell-1}'=S_{\ell-1}'-S_{\ell-2}'=(S_{\ell-1}+S_\ell-h)-S_{\ell-2}=b_{\ell-1}+(S_\ell-h)=b_{\ell-1}+b_\ell-r.
\]

\bigskip
\noindent\emph{Case 3: $\ell\le k\le n-1$ (both terms in the $\Sigma_{\square}$ piece).} 
Here $n-k\le n-\ell$, so $B_{n-k}=\Sigma_{n-k}$ and $B_{n-k-1}=\Sigma_{n-k-1}$, hence,
\[
S_k'=\Sigma_{n-k}-\Sigma_{n-k-1}=S_{k+1}.
\]
At $k=\ell$, using $S_{\ell-1}'=S_{\ell-1}+S_\ell-h$ from Case 2,
\[
b_\ell'=S_\ell'-S_{\ell-1}'=S_{\ell+1}-(S_{\ell-1}+S_\ell-h)=(S_{\ell+1}-S_\ell)+(h-S_{\ell-1})=b_{\ell+1}+r,
\]
while for $\ell+1\le k\le n-1$ we get $b_k'=S_k'-S_{k-1}'=S_{k+1}-S_k=b_{k+1}$.

\bigskip
Assembling the three cases gives
\[
\vb'=(\,b_1,\dots,b_{\ell-2},\ b_{\ell-1}+b_\ell-r,\ b_{\ell+1}+r,\ b_{\ell+2},\dots,b_n\,),
\]
which is the statement for $1<\ell<n$. 

\bigskip
The two boundaries are the degenerate specializations:

\bigskip
\noindent\emph{Boundary $\ell=1$.} 
Then $n-\ell=n-1$, so every $B_m$ with $0\le m\le n-1$ lies in the $\Sigma_{\square}$ piece and only Case 3 occurs.
Hence, $S_k'=S_{k+1}$ for all $k$, giving $b_1'=S_1'=S_2=b_1+b_2$ and $b_k'=S_{k+1}-S_k=b_{k+1}$ for $2\le k\le n-1$; that is, $\vb'=(b_1+b_2,b_3,\dots,b_n)$.

\bigskip 
\noindent\emph{Boundary $\ell=n$.} 
Then $n-\ell=0$, so $B_0=\Sigma_0=0$ is the only term in the $\Sigma_{\square}$ piece and every $B_m$ with $m\ge1$ lies in the $\Sigma_{\square}-h$ piece.
Only Cases 1 and 2 occur: $b_k'=b_k$ for $1\le k\le n-2$, and at $k=n-1=\ell-1$, $b_{n-1}'=b_{n-1}+b_n-r$.
Thus, $\vb'=(b_1,\dots,b_{n-2},b_{n-1}+b_n-r)$.

\bigskip
In every case each entry of $\vb'$ is a positive integer; in particular $b_{\ell-1}+b_\ell-r\ge b_{\ell-1}\ge1$ because $r\le b_\ell$.
Lemma~\ref{lem:convert} now yields $\X_n(\vb)[h]=\X_{n-1}(\vb')$.
\end{proof}

\begin{remark}\label{rem:blocks}
Partition the $S_n$ layers into $n$ \emph{blocks}, where block $\ell$ consists of the layers at heights $S_{\ell-1}+1,\dots,S_\ell$, so that block $\ell$ has $b_\ell$ layers.
By Theorem~\ref{thm:slice}, the first block consists of $b_1$ identical layers, each equal to $\X_{n-1}(b_1+b_2,b_3,\dots,b_n)$. 
For a later block with $2\le\ell<n$, increasing the height by one transfers one unit from the entry $b_{\ell-1}+b_\ell-r$ to the entry $b_{\ell+1}+r$ of $\vb'$, interpolating between adjacent coordinates of $\vb$.
In the final block $\ell=n$ there is no entry $b_{n+1}$, and instead the last entry $b_{n-1}+b_n-r$ decreases by one as the height increases.
\end{remark}

\begin{example}\label{ex:running-slices}
Continuing Example~\ref{ex:running}, let $\vb=(3,2,2)$, so that $(S_1,S_2,S_3)=(3,5,7)$.
In accordance with Proposition~\ref{prop:layers}, the lattice points of $\X_3(3,2,2)$ occur at exactly the heights $x_3\in\{1,\dots,7\}$, so there are seven layers, partitioned into blocks of sizes $3$, $2$, and $2$ as in Remark~\ref{rem:blocks}.
Theorem~\ref{thm:slice} gives:
\[
\begin{array}{c|ccccccc}
h & 1 & 2 & 3 & 4 & 5 & 6 & 7\\\hline
(\ell,r) & (1,1) & (1,2) & (1,3) & (2,1) & (2,2) & (3,1) & (3,2)\\
\vb' & (5,2) & (5,2) & (5,2) & (4,3) & (3,4) & (3,3) & (3,2)
\end{array}
\]
The first block consists of three identical layers, each equal to $\X_2(5,2)$; within the second block, one unit is transferred from the first entry of $\vb'$ to the second as the height increases; and in the final block the last entry shrinks, reflecting the fact that the slices contract as the height approaches its maximum.
Figure~\ref{fig:slice} shows the layer at height $h=4$ inside the polytope, and Figure~\ref{fig:layers} shows all seven layers.
\end{example}

\begin{figure}[ht]
    \centering
    \begin{tikzpicture}[scale=0.72]
\draw[back] (0.580,0.680) -- (0.580,6.680);
\draw[back] (0.580,0.680) -- (6.580,0.680);
\draw[back] (0.580,0.680) -- (-1.940,-1.240);
\fill[subfacet] (0.580,3.680) -- (-1.940,1.760) -- (1.060,1.760) -- (5.320,2.720) -- (6.580,3.680) -- cycle;
\draw[subedge] (0.580,3.680) -- (-1.940,1.760) -- (1.060,1.760) -- (5.320,2.720) -- (6.580,3.680) -- cycle;
\node[subvertex] at (0.580,3.680) {};
\node[subvertex] at (-1.940,1.760) {};
\node[subvertex] at (1.060,1.760) {};
\node[subvertex] at (5.320,2.720) {};
\node[subvertex] at (6.580,3.680) {};
\fill[facet] (-1.940,-1.240) -- (2.060,-1.240) -- (2.060,0.760) -- (0.060,2.760) -- (-1.940,2.760) -- cycle;
\fill[facet] (4.900,1.400) -- (4.900,-0.600) -- (6.580,0.680) -- (6.580,4.680) -- (5.740,4.040) -- cycle;
\fill[facet] (-1.100,5.400) -- (0.900,5.400) -- (3.740,6.040) -- (4.580,6.680) -- (0.580,6.680) -- cycle;
\fill[facet] (2.060,-1.240) -- (4.900,-0.600) -- (4.900,1.400) -- (2.060,0.760) -- cycle;
\fill[facet] (-1.940,2.760) -- (0.060,2.760) -- (0.900,5.400) -- (-1.100,5.400) -- cycle;
\fill[facet] (5.740,4.040) -- (6.580,4.680) -- (4.580,6.680) -- (3.740,6.040) -- cycle;
\fill[facet] (0.060,2.760) -- (2.060,0.760) -- (4.900,1.400) -- (5.740,4.040) -- (3.740,6.040) -- (0.900,5.400) -- cycle;
\draw[edge] (0.580,6.680) -- (4.580,6.680);
\draw[edge] (0.580,6.680) -- (-1.100,5.400);
\draw[edge] (4.580,6.680) -- (6.580,4.680);
\draw[edge] (4.580,6.680) -- (3.740,6.040);
\draw[edge] (6.580,0.680) -- (6.580,4.680);
\draw[edge] (6.580,0.680) -- (4.900,-0.600);
\draw[edge] (6.580,4.680) -- (5.740,4.040);
\draw[edge] (3.740,6.040) -- (5.740,4.040);
\draw[edge] (3.740,6.040) -- (0.900,5.400);
\draw[edge] (5.740,4.040) -- (4.900,1.400);
\draw[edge] (-1.100,5.400) -- (0.900,5.400);
\draw[edge] (-1.100,5.400) -- (-1.940,2.760);
\draw[edge] (0.900,5.400) -- (0.060,2.760);
\draw[edge] (4.900,-0.600) -- (4.900,1.400);
\draw[edge] (4.900,-0.600) -- (2.060,-1.240);
\draw[edge] (4.900,1.400) -- (2.060,0.760);
\draw[edge] (-1.940,-1.240) -- (-1.940,2.760);
\draw[edge] (-1.940,-1.240) -- (2.060,-1.240);
\draw[edge] (-1.940,2.760) -- (0.060,2.760);
\draw[edge] (0.060,2.760) -- (2.060,0.760);
\draw[edge] (2.060,-1.240) -- (2.060,0.760);
\node[vertex] at (0.580,0.680) {};
\node[vertex] at (0.580,6.680) {};
\node[vertex] at (4.580,6.680) {};
\node[vertex] at (6.580,0.680) {};
\node[vertex] at (6.580,4.680) {};
\node[vertex] at (3.740,6.040) {};
\node[vertex] at (5.740,4.040) {};
\node[vertex] at (-1.100,5.400) {};
\node[vertex] at (0.900,5.400) {};
\node[vertex] at (4.900,-0.600) {};
\node[vertex] at (4.900,1.400) {};
\node[vertex] at (-1.940,-1.240) {};
\node[vertex] at (-1.940,2.760) {};
\node[vertex] at (0.060,2.760) {};
\node[vertex] at (2.060,-1.240) {};
\node[vertex] at (2.060,0.760) {};
\node[font=\scriptsize,orange!60!black,left=2pt] at (-1.940,1.760) {$x_3=4$};
\end{tikzpicture}
    \caption{The layer $\X_3(3,2,2)[4]=\X_2(4,3)$ inside $\X_3(3,2,2)$.}
    \label{fig:slice}
\end{figure}
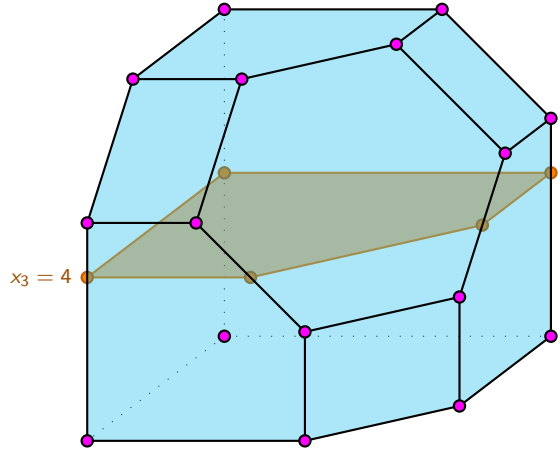

\begin{figure}[ht]
    \centering
    \begin{tikzpicture}[scale=0.20]
\fill[facet] (1,1) -- (7,1) -- (7,5) -- (5,7) -- (1,7) -- cycle;
\draw[edge] (1,1) -- (7,1) -- (7,5) -- (5,7) -- (1,7) -- cycle;
\node[vertex,inner sep=1.05pt] at (1,1) {};
\node[vertex,inner sep=1.05pt] at (7,1) {};
\node[vertex,inner sep=1.05pt] at (7,5) {};
\node[vertex,inner sep=1.05pt] at (5,7) {};
\node[vertex,inner sep=1.05pt] at (1,7) {};
\node[font=\scriptsize] at (4,9.4) {$h=1$};
\node[font=\scriptsize] at (4,-1.7) {$(5,2)$};
\fill[facet] (11,1) -- (17,1) -- (17,5) -- (15,7) -- (11,7) -- cycle;
\draw[edge] (11,1) -- (17,1) -- (17,5) -- (15,7) -- (11,7) -- cycle;
\node[vertex,inner sep=1.05pt] at (11,1) {};
\node[vertex,inner sep=1.05pt] at (17,1) {};
\node[vertex,inner sep=1.05pt] at (17,5) {};
\node[vertex,inner sep=1.05pt] at (15,7) {};
\node[vertex,inner sep=1.05pt] at (11,7) {};
\node[font=\scriptsize] at (14,9.4) {$h=2$};
\node[font=\scriptsize] at (14,-1.7) {$(5,2)$};
\fill[facet] (21,1) -- (27,1) -- (27,5) -- (25,7) -- (21,7) -- cycle;
\draw[edge] (21,1) -- (27,1) -- (27,5) -- (25,7) -- (21,7) -- cycle;
\node[vertex,inner sep=1.05pt] at (21,1) {};
\node[vertex,inner sep=1.05pt] at (27,1) {};
\node[vertex,inner sep=1.05pt] at (27,5) {};
\node[vertex,inner sep=1.05pt] at (25,7) {};
\node[vertex,inner sep=1.05pt] at (21,7) {};
\node[font=\scriptsize] at (24,9.4) {$h=3$};
\node[font=\scriptsize] at (24,-1.7) {$(5,2)$};
\fill[facet] (31,1) -- (37,1) -- (37,4) -- (34,7) -- (31,7) -- cycle;
\draw[edge] (31,1) -- (37,1) -- (37,4) -- (34,7) -- (31,7) -- cycle;
\node[vertex,inner sep=1.05pt] at (31,1) {};
\node[vertex,inner sep=1.05pt] at (37,1) {};
\node[vertex,inner sep=1.05pt] at (37,4) {};
\node[vertex,inner sep=1.05pt] at (34,7) {};
\node[vertex,inner sep=1.05pt] at (31,7) {};
\node[font=\scriptsize] at (34,9.4) {$h=4$};
\node[font=\scriptsize] at (34,-1.7) {$(4,3)$};
\fill[facet] (41,1) -- (47,1) -- (47,3) -- (43,7) -- (41,7) -- cycle;
\draw[edge] (41,1) -- (47,1) -- (47,3) -- (43,7) -- (41,7) -- cycle;
\node[vertex,inner sep=1.05pt] at (41,1) {};
\node[vertex,inner sep=1.05pt] at (47,1) {};
\node[vertex,inner sep=1.05pt] at (47,3) {};
\node[vertex,inner sep=1.05pt] at (43,7) {};
\node[vertex,inner sep=1.05pt] at (41,7) {};
\node[font=\scriptsize] at (44,9.4) {$h=5$};
\node[font=\scriptsize] at (44,-1.7) {$(3,4)$};
\fill[facet] (51,1) -- (56,1) -- (56,3) -- (53,6) -- (51,6) -- cycle;
\draw[edge] (51,1) -- (56,1) -- (56,3) -- (53,6) -- (51,6) -- cycle;
\node[vertex,inner sep=1.05pt] at (51,1) {};
\node[vertex,inner sep=1.05pt] at (56,1) {};
\node[vertex,inner sep=1.05pt] at (56,3) {};
\node[vertex,inner sep=1.05pt] at (53,6) {};
\node[vertex,inner sep=1.05pt] at (51,6) {};
\node[font=\scriptsize] at (54,9.4) {$h=6$};
\node[font=\scriptsize] at (54,-1.7) {$(3,3)$};
\fill[facet] (61,1) -- (65,1) -- (65,3) -- (63,5) -- (61,5) -- cycle;
\draw[edge] (61,1) -- (65,1) -- (65,3) -- (63,5) -- (61,5) -- cycle;
\node[vertex,inner sep=1.05pt] at (61,1) {};
\node[vertex,inner sep=1.05pt] at (65,1) {};
\node[vertex,inner sep=1.05pt] at (65,3) {};
\node[vertex,inner sep=1.05pt] at (63,5) {};
\node[vertex,inner sep=1.05pt] at (61,5) {};
\node[font=\scriptsize] at (64,9.4) {$h=7$};
\node[font=\scriptsize] at (64,-1.7) {$(3,2)$};
\end{tikzpicture}
    \caption{The layers of $\X_3(3,2,2)$; below each layer is its parameter vector $\vb'$.}
    \label{fig:layers}
\end{figure}
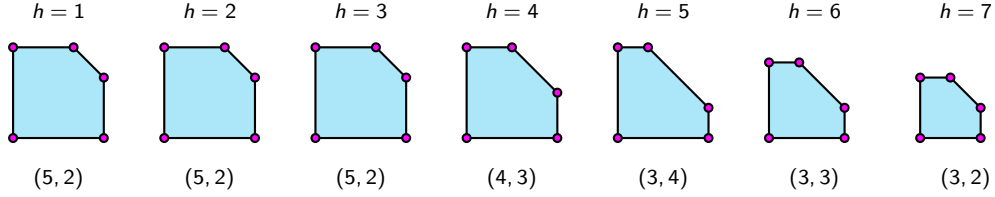

%%%%%%%%%%%%%%%%%%%%%%%%%%%%
\section{The lattice-point count recursion}\label{sec:recursion}

For $\vb=(b_1,\dots,b_n)\in\Z_{>0}^n$, let
\[
L(\vb)\ :=\ \#\bigl(\X_n(\vb)\cap\Z^n\bigr)
\]
denote the number of lattice points of the $\vb$-parking-function polytope; since the length of the argument determines the ambient dimension, this notation is unambiguous.
Equivalently, $L(\vb)=\ehr_{\X_n(\vb)}(1)$ in the notation of Section~\ref{subsec:ehrhart}.
Partitioning the lattice points of $\X_n(\vb)$ by the value of the last coordinate and applying Theorem~\ref{thm:slice} immediately yields the following recursion.
\begin{theorem}\label{thm:recursion}
Let $n\ge2$ and $\vb=(b_1,\dots,b_n)\in\Z_{>0}^n$. Then
\begin{align*}
L(b_1,\dots,b_n)
={}& b_1\,L(b_1+b_2,\,b_3,\dots,b_n)\\
&{}+\sum_{\ell=2}^{n-1}\ \sum_{r=1}^{b_\ell}
   L\bigl(b_1,\dots,b_{\ell-2},\,b_{\ell-1}+b_\ell-r,\,b_{\ell+1}+r,\,b_{\ell+2},\dots,b_n\bigr)\\
&{}+\sum_{r=1}^{b_n} L\bigl(b_1,\dots,b_{n-2},\,b_{n-1}+b_n-r\bigr).
\end{align*}
\end{theorem}

\begin{proof}
By Proposition~\ref{prop:layers}, the lattice points of $\X_n(\vb)$ are partitioned according to the level $x_n=h$ for $1\le h\le S_n$. 
The number of lattice points at level $h$ equals $|\X_n(\vb)[h]\cap\Z^{n-1}|=L(\vb')$, where $\X_n(\vb)[h]=\X_{n-1}(\vb')$ is the layer of Theorem~\ref{thm:slice}.
Summing over $h$ and grouping the levels into the blocks of Remark~\ref{rem:blocks} gives the stated identity, the first block contributing $b_1$ identical copies.
\end{proof}

Together with the initial condition $L(b_1)=b_1$ for $n=1$ (the lattice points of $\X_1(b_1)=[1,b_1]$ are $1,\dots,b_1$), Theorem~\ref{thm:recursion} determines $L(\vb)$ for every $\vb$.
Note that the parameter changes $\vb\mapsto\vb'$ with $\ell<n$ preserve the total weight $b_1+\cdots+b_n$, while those with $\ell=n$ decrease it by $r$; in particular, every argument on the right-hand side has length $n-1$, so the recursion terminates after $n-1$ steps.

\begin{example}\label{ex:running-count}
For $\X_3(3,2,2)$ (Examples~\ref{ex:running} and~\ref{ex:running-slices}), Theorem~\ref{thm:recursion} reads
\[
L(3,2,2)=3\,L(5,2)+\bigl(L(4,3)+L(3,4)\bigr)+\bigl(L(3,3)+L(3,2)\bigr)
=3\cdot46+(43+39)+(30+22)=272,
\]
where the length-two counts are obtained from a further application of the recursion, which, together with the initial condition $L(b_1)=b_1$, yields the closed form $L(b_1,b_2)=(b_1+b_2)^2-\binom{b_2+1}{2}$.
\end{example}

%%%%%%%%%%%%%%%%%%%%%%%%%%%%

\section{The Ehrhart polynomial of \texorpdfstring{$\X_n(\vb)$}{Xn(b)} for arbitrary \texorpdfstring{$\vb$}{b}}
\label{sec:general}

In this section we determine the Ehrhart polynomial of $\X_n(\vb)$ for arbitrary $\vb\in\Z_{>0}^n$, thereby resolving parts (3) and (4) of \cite[Problem~5.8]{HanadaLentferVindas}.
The results of this section were obtained independently of the recent work of Liu and Thawinrak \cite{LiuThawinrak}, where an equivalent formula is established in the more general setting of their polytopes $PF(\mathbf u)$.
We compare the two approaches in Remark~\ref{rem:LT}.
The proof of the main result of this section (Theorem~\ref{thm:general}) proceeds in four steps.
\begin{enumerate}
    \item First, we show that dilating a $\vb$-parking-function polytope produces a lattice translate of another $\vb$-parking-function polytope, so that the Ehrhart function of $\X_n(\vb)$ is a specialization of the lattice-point count $L$ (Lemma~\ref{lem:dilation}).

    \item Second, we show that $L$, initially defined only on positive integer parameter vectors, is the restriction of a single polynomial $\Lp_n$ in the parameters (Proposition~\ref{prop:polynomiality}); combined with the first step, this expresses $\ehr_{\X_n(\vb)}(t)$ as a polynomial in $\vb$ and $t$ (Corollary~\ref{cor:ehrpoly}).

    \item Third, for those $\vb$ whose signed Minkowski-sum decomposition \eqref{eq:signedmink} has nonnegative coefficients, we apply Postnikov's lattice-point formula for trimmed generalized permutahedra \cite[Theorem~11.3]{Postnikov} to the lifted polytope, obtaining an explicit formula as a sum over draconian sequences.

    \item Finally, since both sides of the resulting identity are polynomials in $\vb$, the vanishing principle of Lemma~\ref{lem:vanishing} extends it to every $\vb\in\Z_{>0}^n$, completing the proof of Theorem~\ref{thm:general}.
\end{enumerate}
Throughout, $\mathbbm{1}:=(1,\dots,1)$ denotes the all-ones vector of the appropriate length.

\subsection{Dilation and polynomiality}\label{subsec:dilation}

We begin by recording how dilation acts on a $\vb$-parking-function polytope.
The following identity generalizes the dilation relations of Hanada, Lentfer, and Vindas-Mel\'endez \cite{HanadaLentferVindas} for the two-parameter family, making explicit how dilation transforms an arbitrary parameter vector.

\begin{lemma}\label{lem:dilation}
Let $t$ be a positive integer and set
\[
\vb^{(t)}\ :=\ \bigl(t(b_1-1)+1,\ tb_2,\ tb_3,\ \dots,\ tb_n\bigr)\in\Z_{>0}^n.
\]
Then $t\,\X_n(\vb)=(t-1)\mathbbm{1}+\X_n\bigl(\vb^{(t)}\bigr)$.
In particular,
\[
\ehr_{\X_n(\vb)}(t)\ =\ L\bigl(\vb^{(t)}\bigr).
\]
\end{lemma}

\begin{proof}
By Theorem~\ref{thm:ineq}, $\X_n(\vb)$ is the set of solutions of the full system $x_i\ge1$ ($i\in[n]$) and $\sum_{i\in I}x_i\le\Sigma_{|I|}$ ($\emptyset\ne I\subseteq[n]$), so its dilate $t\,\X_n(\vb)$ is the set of solutions of $x_i\ge t$ and $\sum_{i\in I}x_i\le t\,\Sigma_{|I|}$.
Substituting $\vx=\vz+(t-1)\mathbbm{1}$ transforms this system into $z_i\ge1$ and
\[
\sum_{i\in I}z_i\ \le\ t\,\Sigma_{k}-k(t-1)\ =:\ \Sigma_k'\qquad(k=|I|).
\]
The successive differences of the right-hand sides are $\Sigma_k'-\Sigma_{k-1}'=t\,S_{n-k+1}-(t-1)$, so the numbers $\Sigma_k'$ arise from the partial sums $S_i':=t\,S_i-(t-1)$, whose difference sequence is exactly $\vb^{(t)}$: indeed $b_1'=S_1'=t(b_1-1)+1\ge1$ and $b_i'=S_i'-S_{i-1}'=tb_i$ for $2\le i\le n$.
By Theorem~\ref{thm:ineq} applied to $\vb^{(t)}$, the transformed system cuts out precisely $\X_n(\vb^{(t)})$.
Finally, translation by the lattice vector $(t-1)\mathbbm{1}$ preserves the number of lattice points, so $\ehr_{\X_n(\vb)}(t)=\#\bigl(\X_n(\vb^{(t)})\cap\Z^n\bigr)=L(\vb^{(t)})$.
\end{proof}

\begin{example}\label{ex:dilation}
Take $\vb=(3,2,2)$ and $t=2$, so that $\vb^{(2)}=(2\cdot2+1,\ 2\cdot2,\ 2\cdot2)=(5,4,4)$.
By Example~\ref{ex:running}, the polytope $\X_3(3,2,2)$ is cut out by \[x_i\ge1, \quad x_i\le7, \quad x_i+x_j\le12,\quad \text{ and } \quad x_1+x_2+x_3\le15,\] so its dilate $2\,\X_3(3,2,2)$ is cut out by \[x_i\ge2, \quad x_i\le14, \quad x_i+x_j\le24, \quad \text{ and } \quad x_1+x_2+x_3\le30.\]
Substituting $\vx=\vz+\mathbbm 1$ transforms these into \[z_i\ge1, \quad z_i\le13, \quad z_i+z_j\le22, \text{ and } \quad z_1+z_2+z_3\le27,\] which are exactly the inequalities of $\X_3(5,4,4)$: its partial sums are $(S_1,S_2,S_3)=(5,9,13)$, giving $(\Sigma_1,\Sigma_2,\Sigma_3)=(13,22,27)$.
Hence, $2\,\X_3(3,2,2)=\mathbbm 1+\X_3(5,4,4)$, and a lattice-point count in SageMath confirms $\ehr_{\X_3(3,2,2)}(2)=L(5,4,4)=1743$.
\end{example}

To make use of Lemma~\ref{lem:dilation} we show that $L$ is a polynomial function of the parameter vector.
We will use the following standard polynomial-vanishing principle twice.

\begin{lemma}\label{lem:vanishing}
If a polynomial $f\in\Q[x_1,\dots,x_n]$ vanishes at every point of $\Z_{>0}^n$, then $f=0$.
\end{lemma}

\begin{proposition}\label{prop:polynomiality}
There is a unique polynomial $\Lp_n\in\Q[x_1,\dots,x_n]$ of degree at most $n$ such that
\[
L(\vb)=\Lp_n(b_1,\dots,b_n)\qquad\text{for all }\vb\in\Z_{>0}^n.
\]
\end{proposition}

Parameter polynomiality also follows from the signed generalized-permutahedron valuation theory of Jochemko and Ravichandran \cite[Proposition~4.7 and Corollary~4.8]{JochemkoRavichandran}, applied to the decomposition of Bayer \etal \cite{BBCDDDLMMNV}, whose simplex coefficients are affine-linear in $\vb$; the proof below instead gives a direct, self-contained derivation from the slice recursion of Theorem~\ref{thm:recursion}.

\begin{proof}
Uniqueness is immediate from Lemma~\ref{lem:vanishing}, applied to the difference of two candidates.
For existence we induct on $n$.
For $n=1$ we may take $\Lp_1(x_1)=x_1$.
For $n\ge2$, recall the elementary fact that for any polynomial $p\in\Q[x_1,\dots,x_m,r]$ there is a polynomial $\widetilde p\in\Q[x_1,\dots,x_m,y]$ with $\deg\widetilde p\le\deg p+1$ such that $\sum_{r=1}^{y}p(\vx,r)=\widetilde p(\vx,y)$ for every positive integer $y$.
Indeed, expanding $p$ in powers of $r$ and summing term by term, this follows from Faulhaber's formula, which expresses the power sum $\sum_{r=1}^{y}r^d$ as a polynomial in $y$ of degree $d+1$ \cite[Eq.~(6.78)]{GKP}; summation over $r$ therefore raises the degree by at most one.
For the history of Faulhaber's formula, see \cite{KnuthFaulhaber}.
Now apply the recursion of Theorem~\ref{thm:recursion}.
By the induction hypothesis, each evaluation of $L$ at a length-$(n-1)$ vector on the right-hand side is the evaluation of $\Lp_{n-1}$ at an argument depending affinely on $(b_1,\dots,b_n)$ and on the summation index $r$; hence, each inner summand is a polynomial in $(b_1,\dots,b_n,r)$ of degree at most $n-1$.
Summing over $1\le r\le b_\ell$ (respectively $1\le r\le b_n$) and using the Faulhaber fact, each block contributes a polynomial in $(b_1,\dots,b_n)$ of degree at most $n$, as does the first term $b_1\,\Lp_{n-1}(b_1+b_2,b_3,\dots,b_n)$.
The sum of these polynomials is the desired $\Lp_n$.
\end{proof}

\begin{example}\label{ex:polynomiality}
Example~\ref{ex:running-count} records the length-two count $\Lp_2(x_1,x_2)=(x_1+x_2)^2-\binom{x_2+1}{2}$; here we revisit it to illustrate the polynomiality mechanism of the proof.
From the base case $\Lp_1(x_1)=x_1$, the $n=2$ instance of Theorem~\ref{thm:recursion} writes \[L(b_1,b_2)=b_1\,\Lp_1(b_1+b_2)+\sum_{r=1}^{b_2}\Lp_1(b_1+b_2-r)\] as a sum of values of $\Lp_1$ at arguments affine in $(b_1,b_2,r)$.
Carrying out the sum over $r$ by Faulhaber's formula turns this into a polynomial and raises the degree by exactly one, from $\deg\Lp_1=1$ to $\deg\Lp_2=2=n$; the same step drives the induction in general, and one more application yields the degree-three polynomial $\Lp_3$ of Example~\ref{ex:running-ehrhart}.
\end{example}

\begin{corollary}\label{cor:ehrpoly}
For every $\vb\in\Z_{>0}^n$,
\[
\ehr_{\X_n(\vb)}(t)\ =\ \Lp_n\bigl(t(b_1-1)+1,\ tb_2,\ \dots,\ tb_n\bigr).
\]
In particular, each coefficient of $\ehr_{\X_n(\vb)}(t)$ is a polynomial function of $\vb$, and, provided $\vb\ne(1)$ (so that $\X_n(\vb)$ is full-dimensional),
\[
\vol\X_n(\vb)\ =\ \Lp_n^{\mathrm{top}}(b_1-1,\,b_2,\,\dots,\,b_n),
\]
where $\Lp_n^{\mathrm{top}}$ denotes the degree-$n$ homogeneous component of $\Lp_n$.
In the excluded case, $\X_1(1)=\{(1)\}$ is a point, of relative volume $1$.
\end{corollary}

\begin{proof}
The displayed identity combines Lemma~\ref{lem:dilation} with Proposition~\ref{prop:polynomiality}; both sides are polynomials in $t$ agreeing at every positive integer.
For the volume, recall that $\vol\X_n(\vb)=\lim_{t\to\infty}t^{-n}\ehr_{\X_n(\vb)}(t)$ is the coefficient of $t^n$ in the Ehrhart polynomial.
Write $\Lp_n=\Lp_n^{\mathrm{top}}+g$, where $\Lp_n^{\mathrm{top}}$ is the degree-$n$ homogeneous part of $\Lp_n$ and $\deg g\le n-1$.
Each argument $t(b_1-1)+1,\,tb_2,\dots,tb_n$ is affine in $t$, so $g$ contributes only terms of degree at most $n-1$ in $t$, while the homogeneity of $\Lp_n^{\mathrm{top}}$ gives
\[
\Lp_n^{\mathrm{top}}\bigl(t(b_1-1)+1,\,tb_2,\dots,tb_n\bigr)
=t^{\,n}\,\Lp_n^{\mathrm{top}}(b_1-1,\,b_2,\dots,b_n)+O\bigl(t^{\,n-1}\bigr),
\]
since scaling all arguments by $t$ pulls out $t^n$ and the constant $+1$ in the first argument affects only lower-order terms in $t$.
Hence, the coefficient of $t^n$ is $\Lp_n^{\mathrm{top}}(b_1-1,\,b_2,\dots,b_n)$, as claimed.
\end{proof}

\begin{example}\label{ex:cor-ehrpoly}
For $n=2$ and $\vb=(3,2)$, Lemma~\ref{lem:dilation} gives $\vb^{(t)}=(2t+1,\,2t)$, so Corollary~\ref{cor:ehrpoly}, together with the polynomial $\Lp_2$ of Example~\ref{ex:polynomiality}, yields
\[
\ehr_{\X_2(3,2)}(t)=\Lp_2(2t+1,\,2t)=(4t+1)^2-\binom{2t+1}{2}=14t^2+7t+1.
\]
The leading coefficient is the volume, obtained as in the proof of Corollary~\ref{cor:ehrpoly} by evaluating the top component at $(b_1-1,b_2)=(2,2)$: since $\Lp_2^{\mathrm{top}}(x_1,x_2)=x_1^2+2x_1x_2+\tfrac12 x_2^2$, we get $\vol\X_2(3,2)=\Lp_2^{\mathrm{top}}(2,2)=14$.
The substitution uses $t(b_1-1)+1$ and the volume is evaluated at $b_1-1$, not $b_1$: replacing $b_1-1$ by $b_1$ would give the incorrect value $\Lp_2^{\mathrm{top}}(3,2)=23$.
\end{example}

\subsection{The explicit formula}\label{subsec:formula}

To derive an explicit formula, we first record a Minkowski decomposition of the lifted polytope.
Following \cite[Section~2.2]{BBCDDDLMMNV}, let $\overline{\X}_n(\vb)$ denote the lifting of $\X_n(\vb)$ into the hyperplane $\{\vx\in\R^{n+1}:\sum_{i=1}^{n+1}x_i=\sum_{k=1}^n S_k\}$ obtained by appending the slack coordinate $x_{n+1}=\sum_{k=1}^n S_k-\sum_{i=1}^n x_i$.
This map is a unimodular equivalence, so $\overline{\X}_n(\vb)$ and $\X_n(\vb)$ have the same Ehrhart polynomial.
Write $\triangle_I:=\conv\{\ve_i:i\in I\}$ for the coordinate simplex on $I\subseteq[n+1]$.
Bayer \etal proved that $\overline{\X}_n(\vb)$ is a \emph{signed} Minkowski sum of coordinate simplices with explicit coefficients determined by $\vb$~\cite[Proposition~2.14]{BBCDDDLMMNV}; since those coefficients depend only on the cardinality of the index set, their result may be recast as follows.
For $1\le k\le n$, define
\begin{equation}\label{eq:wk}
w_1(\vb):=b_1-1,
\qquad
w_k(\vb):=\sum_{j=0}^{k-2}(-1)^{j}\binom{k-2}{j}\,b_{k-j}\quad(2\le k\le n),
\end{equation}
so that $w_k(\vb)$ is, for $k\ge2$, the $(k-2)$-th finite difference of the sequence $(b_2,b_3,\dots,b_n)$ at its initial term; for instance $w_2=b_2$, $w_3=b_3-b_2$, and $w_4=b_4-2b_3+b_2$.
Then \cite[Proposition~2.14]{BBCDDDLMMNV} states that
\begin{equation}\label{eq:signedmink}
\overline{\X}_n(\vb)
=\sum_{i=1}^n\triangle_{\{i\}}
\ +\!\!\sum_{\emptyset\ne A\subseteq[n]}\!\! w_{|A|}(\vb)\,\triangle_{A\cup\{n+1\}}
\end{equation}
as a signed Minkowski sum.
Note that the map $\vb\mapsto(w_1,\dots,w_n)$ is a triangular affine bijection of $\Z^n$: by \eqref{eq:wk}, $w_1-b_1=-1$ and, for $k\ge2$, $w_k-b_k$ is a linear function of $b_2,\dots,b_{k-1}$ alone, so the linear part of the map is unitriangular and $\vb$ may be recovered from $\mathbf w$ by back-substitution, giving the inverse $b_1=w_1+1$ and $b_k=\sum_{j=2}^{k}\binom{k-2}{j-2}w_j$ for $k\ge2$.
In particular, the inverse has nonnegative coefficients, so every $(w_1,\dots,w_n)\in\Z_{>0}^n$ arises from some $\vb\in\Z_{>0}^n$.

Our formula is indexed by sequences satisfying the following classical condition of Hall type, called \emph{draconian} (more precisely, \emph{G-draconian}) by Postnikov \cite[Definition~9.2]{Postnikov}.

\begin{definition}\label{def:draconian}
Let $\mathcal D(n)$ be the set of all sequences $\alpha=(\alpha_A)$ of nonnegative integers, indexed by the nonempty subsets $A\subseteq[n]$, such that
\begin{equation}\label{eq:hallgen}
\sum_{A\in E}\alpha_A\ \le\ \Bigl|\,\bigcup_{A\in E}A\,\Bigr|
\qquad\text{for every nonempty collection $E$ of nonempty subsets of $[n]$.}
\end{equation}
\end{definition}

By Hall's marriage theorem, $\alpha\in\mathcal D(n)$ if and only if the multiset consisting of $\alpha_A$ copies of each $A$ admits a system of distinct representatives.
Taking $E$ to be the collection of all nonempty subsets in \eqref{eq:hallgen} shows that $\sum_A\alpha_A\le n$, so $\mathcal D(n)$ is finite.

\begin{theorem}[{compare with \cite[Corollary~7.6]{LiuThawinrak}}]\label{thm:general}
For every $\vb\in\Z_{>0}^n$,
\begin{equation}\label{eq:generalsum}
\ehr_{\X_n(\vb)}(t)
\ =\ \sum_{\alpha\in\mathcal D(n)}\ \prod_{\emptyset\ne A\subseteq[n]}
\binom{w_{|A|}(\vb)\,t+\alpha_A-1}{\alpha_A},
\end{equation}
where each factor is understood as the polynomial $\binom{y+\alpha-1}{\alpha}=\frac{y(y+1)\cdots(y+\alpha-1)}{\alpha!}$ in $y=w_{|A|}(\vb)\,t$.
\end{theorem}

\begin{proof}
Both sides of \eqref{eq:generalsum} are polynomial functions of $(b_1,\dots,b_n,t)$.
The right-hand side of \eqref{eq:generalsum}, being a finite sum of products of the polynomials $\binom{w_{|A|}(\vb)\,t+\alpha_A-1}{\alpha_A}$, is visibly polynomial, while the left-hand side, the Ehrhart function $\ehr_{\X_n(\vb)}(t)$, is polynomial in $(b_1,\dots,b_n,t)$ by Corollary~\ref{cor:ehrpoly}.
For $n=1$ the identity is immediate: $\mathcal D(1)=\bigl\{(\alpha_{\{1\}}):\alpha_{\{1\}}\in\{0,1\}\bigr\}$ and the right-hand side is $1+w_1t=(b_1-1)t+1=\ehr_{[1,b_1]}(t)$.
So assume $n\ge2$.

First suppose that $\vb$ lies in the set
\[
D\ :=\ \bigl\{\vb\in\Z_{>0}^n:\ w_k(\vb)\ge 1\ \text{for } 1\le k\le n\bigr\},
\]
and let $t$ be a positive integer.
Dilating \eqref{eq:signedmink} gives
\[
t\,\overline{\X}_n(\vb)= t(\ve_1+\cdots+\ve_n)+P,
\qquad
P:=\!\!\sum_{\emptyset\ne A\subseteq[n]}\!\!\bigl(t\,w_{|A|}(\vb)\bigr)\triangle_{A\cup\{n+1\}},
\]
where the coefficients are now nonnegative integers, so that $P$ is a genuine (unsigned) Minkowski sum, that is, a type-$\mathcal Y$ generalized permutahedron.
Discarding the lattice translation, $\ehr_{\X_n(\vb)}(t)=\#\bigl(P\cap\Z^{n+1}\bigr)$.
Let
\[
\mathcal Q\ :=\ \triangle_{[n+1]}+P\ =\ \bigl(1+tw_n\bigr)\triangle_{[n+1]}\ +\!\!\sum_{\emptyset\ne A\subsetneq[n]}\!\!\bigl(tw_{|A|}\bigr)\triangle_{A\cup\{n+1\}},
\]
where we abbreviate $w_k=w_k(\vb)$ and have merged the summand $\triangle_{[n+1]}$ with the term $A=[n]$.
The trimmed generalized permutahedron of $\mathcal Q$, in the sense of \cite[Definition~11.2]{Postnikov}, is $\mathcal Q^-=P$, since trimming removes one summand $\triangle_{[n+1]}$ from the Minkowski sum \cite[Lemma~11.1]{Postnikov}.
Postnikov's lattice-point formula \cite[Theorem~11.3]{Postnikov}, applied to $\mathcal Q$ on the ground set $[n+1]$, now yields
\begin{equation}\label{eq:postapplied}
\#\bigl(P\cap\Z^{n+1}\bigr)
=\sum_{(c,\beta)}\binom{tw_n+c}{c}\prod_{\emptyset\ne A\subsetneq[n]}\binom{tw_{|A|}+\beta_A-1}{\beta_A},
\end{equation}
where the sum ranges over the draconian data for $\mathcal Q$: nonnegative integers $c$ and $\beta=(\beta_A)_{\emptyset\ne A\subsetneq[n]}$ such that $c+\sum_A\beta_A=n$ and, for every subfamily of the index sets, the number of chosen members (with multiplicity) is at most the cardinality of their union minus one.
For a subfamily involving the index set $[n+1]$ this condition reads $c+\sum_{A\in E}\beta_A\le n$, which already follows from the total count $c+\sum_A\beta_A=n$; for a subfamily avoiding $[n+1]$ it reads \[\sum_{A\in E}\beta_A\le\bigl|\bigcup_{A\in E}A\cup\{n+1\}\bigr|-1=\bigl|\bigcup_{A\in E}A\bigr|,\] which is condition \eqref{eq:hallgen} restricted to proper subsets.
Thus, the admissible data are exactly the sequences $\beta$ satisfying \eqref{eq:hallgen} for collections of proper subsets, with $c:=n-\sum_A\beta_A$; note that $c\ge0$ holds automatically, by \eqref{eq:hallgen} applied to the collection of all nonempty proper subsets of $[n]$, whose union is $[n]$.
Expanding each first factor by the hockey-stick identity
\[
\binom{tw_n+c}{c}=\sum_{\alpha_{[n]}=0}^{c}\binom{tw_n+\alpha_{[n]}-1}{\alpha_{[n]}}
\]
converts the sum \eqref{eq:postapplied} into a sum over pairs $(\beta,\alpha_{[n]})$ with $\alpha_{[n]}\le n-\sum_A\beta_A$.
Setting $\alpha_A:=\beta_A$ for $A\subsetneq[n]$, such pairs correspond bijectively to the sequences $\alpha\in\mathcal D(n)$: conditions \eqref{eq:hallgen} for collections $E$ of proper subsets are those imposed on $\beta$, while for collections $E$ containing $[n]$ the condition $\sum_{A\in E}\alpha_A\le n$ follows from $\alpha_{[n]}+\sum_{A\subsetneq[n]}\beta_A\le n$.
This proves \eqref{eq:generalsum} for all $\vb\in D$ and all positive integers $t$.

Finally, we remove the restriction on $\vb$, and fix a positive integer $t$.
As noted at the start of the proof, each side of \eqref{eq:generalsum} is then a polynomial function of $(b_1,\dots,b_n)$; let $g$ denote their difference.
By the previous paragraph, $g$ vanishes on $D$.
The triangular affine map $\vb\mapsto(w_1,\dots,w_n)$ described after \eqref{eq:signedmink} restricts to a bijection from $D$ onto $\Z_{>0}^n$, so composing $g$ with the inverse map $\mathbf{w}\mapsto\vb$ produces a polynomial vanishing at every point of $\Z_{>0}^n$; by Lemma~\ref{lem:vanishing} that polynomial is zero, and hence, $g=0$.
It follows that \eqref{eq:generalsum} holds for every $\vb\in\Z_{>0}^n$ and every positive integer $t$, and therefore, since each side is a polynomial in $t$, as an identity of polynomials.
\end{proof}

\begin{example}\label{ex:n2}
Let $n=2$ and write $w_1=b_1-1$, $w_2=b_2$.
The set $\mathcal D(2)$ consists of the eight sequences $(\alpha_{\{1\}},\alpha_{\{2\}},\alpha_{\{1,2\}})$ given by $(0,0,0)$, $(1,0,0)$, $(0,1,0)$, $(1,1,0)$, $(0,0,1)$, $(1,0,1)$, $(0,1,1)$, and $(0,0,2)$, so Theorem~\ref{thm:general} gives
\[
\ehr_{\X_2(b_1,b_2)}(t)=1+2w_1t+(w_1t)^2+w_2t+2w_1w_2t^2+\binom{w_2t+1}{2}.
\]
For $\vb=(1,1)$, this is $1+t+\tbinom{t+1}{2}=\tfrac{(t+1)(t+2)}{2}$, the Ehrhart polynomial of the unimodular triangle $PF_2=\X_2(1,1)$.
\end{example}

Extracting the leading coefficient of \eqref{eq:generalsum} gives an explicit formula for the volume, which for $\vb\in D$ is an instance of Postnikov's volume formula \cite[Theorem~9.3]{Postnikov}.

\begin{corollary}\label{cor:volume}
For every $\vb\in\Z_{>0}^n$ with $\vb\ne(1)$,
\[
\vol\X_n(\vb)\ =\ \sum_{\substack{\alpha\in\mathcal D(n)\\ \sum_A\alpha_A=n}}\ \prod_{\emptyset\ne A\subseteq[n]}\frac{w_{|A|}(\vb)^{\alpha_A}}{\alpha_A!}.
\]
\end{corollary}

\begin{proof}
Divide \eqref{eq:generalsum} by $t^n$ and let $t\to\infty$: the term indexed by $\alpha$ has degree $\sum_A\alpha_A\le n$ in $t$, with equality contributing leading coefficient $\prod_A w_{|A|}^{\alpha_A}/\alpha_A!$.
\end{proof}

\begin{example}\label{ex:running-ehrhart}
We illustrate the results of this section in dimension three.
Iterating the recursion of Theorem~\ref{thm:recursion} as in the proof of Proposition~\ref{prop:polynomiality} produces the polynomial
\begin{align*}
\Lp_3\ ={}& b_1^3+3b_1^2b_2+3b_1^2b_3+3b_1b_2^2+6b_1b_2b_3+\tfrac32\,b_1b_3^2+\tfrac56\,b_2^3+2b_2^2b_3+b_2b_3^2+\tfrac16\,b_3^3\\
&-\tfrac32\,b_1b_3-\tfrac12\,b_2^2-2b_2b_3-\tfrac12\,b_3^2-\tfrac13\,b_2+\tfrac13\,b_3,
\end{align*}
which returns $\Lp_3(3,2,2)=272$, as in Example~\ref{ex:running-count}.
For the running example $\vb=(3,2,2)$, Lemma~\ref{lem:dilation} gives $\vb^{(t)}=(2t+1,\,2t,\,2t)$, and Corollary~\ref{cor:ehrpoly} yields
\[
\ehr_{\X_3(3,2,2)}(t)\ =\ \Lp_3(2t+1,\,2t,\,2t)\ =\ 172t^3+84t^2+15t+1.
\]
The same polynomial arises from Theorem~\ref{thm:general}: here the coefficients \eqref{eq:wk} are $(w_1,w_2,w_3)=(2,2,0)$, and of the $84$ draconian sequences in $\mathcal D(3)$ only those supported on subsets of size at most two contribute to \eqref{eq:generalsum}, because $w_3=0$.
By Corollary~\ref{cor:volume}, $\vol\X_3(3,2,2)=172$.
To illustrate the case of mixed signs, take instead $\vb=(2,3,1)$: then $(w_1,w_2,w_3)=(1,3,-2)$, and \eqref{eq:generalsum} nevertheless yields the correct Ehrhart polynomial
\[
\ehr_{\X_3(2,3,1)}(t)\ =\ \Lp_3(t+1,\,3t,\,t)\ =\ \tfrac{619}{6}t^3+61t^2+\tfrac{77}{6}t+1,
\]
with $\ehr_{\X_3(2,3,1)}(1)=L(2,3,1)=178$.
\end{example}

\begin{remark}
The existence of a lattice-point formula for $\X_n(\vb)$ via generalized permutahedra was already observed in \cite{BBCDDDLMMNV}, where the number of lattice points is noted to be computable from a result of Jochemko and Ravichandran \cite[Corollary~4.8]{JochemkoRavichandran}.
Theorem~\ref{thm:general} makes the enumeration explicit in the parameters $\vb$ and $t$ simultaneously, and Lemma~\ref{lem:dilation} explains why a single polynomial in $\vb$ governs the entire Ehrhart theory of the family.
\end{remark}

\begin{remark}\label{rem:LT}
Liu and Thawinrak \cite{LiuThawinrak} study the polytopes $PF(\mathbf u)$, defined for any nondecreasing $\mathbf u=(u_1,\dots,u_n)\in\R_{\ge0}^n$ as the convex hull of all $\mathbf u$-parking functions; in their notation, $\X_n(\vb)=\mathbbm{1}+PF(S_1-1,\,S_2-1,\,\dots,\,S_n-1)$, and conversely every $PF(\mathbf u)$ with $\mathbf u$ a strictly increasing vector of nonnegative integers is a lattice translate of some $\X_n(\vb)$.
Their Corollary~7.6 is the analogue of Theorem~\ref{thm:general} for integral $PF(\mathbf u)$, obtained from the counterpart of the decomposition \eqref{eq:signedmink} \cite[Proposition~7.3]{LiuThawinrak} by applying Postnikov's formula \cite[Theorem~11.3]{Postnikov}.
We note that Postnikov's theorem is stated for nonnegative Minkowski coefficients, whereas the coefficients $w_k$ (equivalently, their $y_I$) may be negative.
The extension to signed Minkowski sums is already provided in full generality by Jochemko and Ravichandran \cite[Proposition~4.7 and Corollary~4.8]{JochemkoRavichandran}; the passage through Lemma~\ref{lem:dilation}, Proposition~\ref{prop:polynomiality}, and Lemma~\ref{lem:vanishing} in the proof of Theorem~\ref{thm:general} gives an independent, family-specific derivation of the same extension from the slice recursion.
For the volume, Liu and Thawinrak establish the counterpart of Corollary~\ref{cor:volume} \cite[Corollary~7.7]{LiuThawinrak} and extend it to non-integral $\mathbf u$ by polynomiality of mixed volumes.
\end{remark}

%%%%%%%%%%%%%%%%%%%%%%%%%%%%

\section{The Ehrhart polynomial of \texorpdfstring{$\X_n(a,b)$}{Xn(a,b)}}
\label{sec:ehrhart}

Throughout this section we specialize to $\vb=(a,b,\dots,b)\in\Z_{>0}^n$, using the abbreviation $\X_n(a,b)$ introduced in Section~\ref{sec:intro}; these are the generalized parking-function polytopes studied by Hanada, Lentfer, and Vindas-Mel\'endez~\cite{HanadaLentferVindas}.
For these polytopes the partial sums specialize to $S_i=a+(i-1)b$ in the inequality description of Theorem~\ref{thm:ineq}.
We retain from Section~\ref{sec:general} the lifting $\overline{\X}_n(a,b)$, the coordinate simplices $\triangle_I$, and the coefficients $w_k$ of \eqref{eq:wk}.

We obtain the full Ehrhart polynomial of $\X_n(a,b)$ in closed form (Theorem~\ref{thm:ehrhart-ab}) by adapting the method Behrend used for partial permutahedra in \cite{Behrend}.
The proof of Theorem~\ref{thm:ehrhart-ab} proceeds in four steps.
\begin{enumerate}
    \item First, we observe that for $\vb=(a,b,\dots,b)$ the signed Minkowski sum \eqref{eq:signedmink} has nonnegative coefficients, and is thus an unsigned Minkowski sum of dilated coordinate simplices, all of dimension at most two (Corollary \ref{lem:minkowski}); this exhibits $\overline{\X}_n(a,b)$ as a type-$\mathcal Y$ generalized permutahedron.

    \item Second, the general formula of Theorem~\ref{thm:general} collapses to a sum over sequences supported on subsets of size at most two, which Behrend's bijection reindexes as a weighted sum over the graphs $\mathcal G(n)$ whose components each carry at most one cycle (Proposition \ref{prop:graphsum}).

    \item Third, since the weight is multiplicative over
connected components, the exponential formula converts this into a generating function $\exp(f)$, and summing $f$ over the four types of connected graphs (trees, looped trees, enhanced trees, quasitrees) via the tree function $T$ closes
it into \eqref{eq:gf-ab}.

\item Finally, a Lambert-$W$ coefficient extraction turns the generating function into the explicit polynomial \eqref{eq:ehrhart-ab}.
\end{enumerate}

\begin{corollary}[{of \cite[Proposition~2.14]{BBCDDDLMMNV}}]\label{lem:minkowski}
For $\vb=(a,b,\dots,b)$, the lifted polytope is the Minkowski sum
\[
\overline{\X}_n(a,b)
=\sum_{i=1}^n\triangle_{\{i\}}
+(a-1)\sum_{i=1}^n\triangle_{\{i,\,n+1\}}
+b\!\!\sum_{1\le i<j\le n}\!\!\triangle_{\{i,j,\,n+1\}}.
\]
In particular, $\overline{\X}_n(a,b)$ is a type-$\mathcal Y$ generalized permutahedron: a Minkowski sum $\sum_I y_I\triangle_I$ with all $y_I\ge 0$.
\end{corollary}

\begin{proof}
For $\vb=(a,b,\dots,b)$, the coefficients \eqref{eq:wk} are $w_1=a-1$ and $w_2=b_2=b$, while $w_k=0$ for $3\le k\le n$, because $w_k$ is the $(k-2)$-nd finite difference of the constant sequence $(b,b,\dots,b)$.
Substituting into \eqref{eq:signedmink} gives the stated decomposition, and all coefficients are nonnegative because $a,b\ge1$.
\end{proof}

\begin{remark}
    When $b=1$ and $a=p-n+2$, Corollary~\ref{lem:minkowski} is the decomposition of the lifted partial permutahedron $\widetilde{\mathcal P}(n,p)$ in \cite[Eq.~(6)]{Behrend}; indeed $\X_n(a,1)\cong\mathcal P(n,n+a-2)$ \cite[Proposition~3.16]{HanadaLentferVindas}.
\end{remark}

We now run Behrend's argument with the coefficients $a-1$ and $b$ in place of the coefficients $p-n+1$ and $1$ for $\mathcal P(n,p)$, respectively.
Recall the tree function 
\begin{equation}\label{eq:treefn}
T(z):=\sum_{m\ge1}m^{m-1}\frac{z^m}{m!}=-W_0(-z),\qquad T(z)=z\,e^{T(z)},
\end{equation}
where $W_0$ is the principal branch of the Lambert $W$ function; this is the function $g_b$ of \cite[Lemma 3.26]{HanadaLentferVindas} with parameter equal to $1$.
Let $\mathcal{G}(n)$ be the set of graphs on $[n]$ (loops and multiple edges are allowed) in which every connected component contains at most one cycle, where a loop is a $1$-cycle and a pair of parallel edges a $2$-cycle. 
For $G\in \mathcal{G}(n)$, an \emph{edge pair} is a pair of parallel edges and a \emph{single edge} is a non-loop edge outside any edge pair. 

\begin{proposition}\label{prop:graphsum}
For all positive integers $a,b,n$,
\begin{equation}\label{eq:graphsum}
\ehr_{\X_n(a,b)}(t)
=\sum_{G\in\mathcal G(n)}
\bigl((a-1)t\bigr)^{\#\mathrm{loops}(G)}(bt)^{\#\mathrm{single}(G)}
\left(\frac{bt(bt+1)}{2}\right)^{\#\mathrm{pairs}(G)}.
\end{equation}
\end{proposition}

\begin{proof}
Let $\mathcal E(n)$ denote the collection of all subsets of $[n]$ of size one or two, and let $\mathcal A(n)$ be the set of sequences $\alpha=(\alpha_S)_{S\in\mathcal E(n)}$ of nonnegative integers satisfying the Hall-type condition
\begin{equation}\label{eq:hall}
\sum_{S\in E}\alpha_S\ \le\ \Bigl|\,\bigcup_{S\in E}S\,\Bigr|
\qquad\text{for all nonempty } E\subseteq\mathcal E(n).
\end{equation}
We specialize Theorem~\ref{thm:general} to $\vb=(a,b,\dots,b)$.
As in the proof of Corollary~\ref{lem:minkowski}, we have $w_1=a-1$, $w_2=b$, and $w_k=0$ for $k\ge3$.
Since $\binom{0\cdot t+\alpha-1}{\alpha}=0$ for $\alpha\ge1$, every $\alpha\in\mathcal D(n)$ with $\alpha_A\ge1$ for some $|A|\ge3$ contributes nothing to \eqref{eq:generalsum}, so the sum may be restricted to sequences supported on $\mathcal E(n)$; for such sequences the condition \eqref{eq:hallgen} reduces to \eqref{eq:hall}, because enlarging a collection $E$ by sets $A$ with $\alpha_A=0$ only weakens the constraint.
Hence,
\begin{equation}\label{eq:binomsum}
\ehr_{\X_n(a,b)}(t)
=\sum_{\alpha\in\mathcal A(n)}\ \prod_{i=1}^{n}\binom{(a-1)t+\alpha_{\{i\}}-1}{\alpha_{\{i\}}}\prod_{1\le i<j\le n}\binom{bt+\alpha_{\{i,j\}}-1}{\alpha_{\{i,j\}}};
\end{equation} 
for $b=1$ and $a=p-n+2$ this recovers \cite[Eq.~(15)]{Behrend}, whose coefficients $p-n+1$ and $1$ are replaced here by $a-1$ and $b$.
Taking $E=\{\{i\}\}$ and $E=\{\{i,j\}\}$ in \eqref{eq:hall} shows that $\alpha_{\{i\}}\in\{0,1\}$ and $\alpha_{\{i,j\}}\in\{0,1,2\}$ for every $\alpha\in\mathcal A(n)$, so each binomial factor in \eqref{eq:binomsum} equals $1$ or $(a-1)t$ in the first product, and $1$, $bt$, or $bt(bt+1)/2$ in the second.

Finally, let $\Gamma(\alpha)$ be the graph on $[n]$ with $\alpha_{\{i\}}$ loops at vertex $i$ and $\alpha_{\{i,j\}}$ edges joining $i$ and $j$.
By Hall's marriage theorem, condition \eqref{eq:hall} holds if and only if the family of edges of $\Gamma(\alpha)$ (each edge recorded as its set of endpoints) admits a system of distinct representatives, that is, if and only if $\Gamma(\alpha)$ admits an orientation in which every vertex has indegree at most one; and a graph admits such an orientation precisely when each of its connected components contains at most one cycle.
Hence, $\Gamma$ is a bijection from $\mathcal A(n)$ onto $\mathcal G(n)$; see \cite[Section~4.2]{Behrend} for a detailed verification.
Under $\Gamma$, the indices with $\alpha_{\{i\}}=1$ correspond to loops, those with $\alpha_{\{i,j\}}=1$ to single edges, and those with $\alpha_{\{i,j\}}=2$ to edge pairs, so \eqref{eq:binomsum} becomes \eqref{eq:graphsum}.
\end{proof}

\begin{theorem}\label{thm:ehrhart-ab}
For all positive integers $a,b,n$,
\begin{equation}\label{eq:ehrhart-ab}
\begin{split}
\ehr_{\X_n(a,b)}(t)
=\frac{1}{2^{\,n}}\sum_{i=0}^{\lfloor n/2\rfloor}\sum_{j=2i}^{n}
(-1)^{i+1}&\binom{n}{\,n-j,\ j-2i,\ i,\ i\,}\,i!\,(2j-4i-3)!!\\[-1pt]
&\times(bt)^{\,j-i}\bigl((2bn-b+2a-2)t+2\bigr)^{\,n-j},
\end{split}
\end{equation}
where $(2j-4i-3)!!:=-\prod_{k=1}^{\,j-2i}(2k-3)$; this convention extends the double factorial to negative odd arguments, with $(-1)!!=1$ and $(-3)!!=-1$, precisely so that the boundary terms with $j=2i$ are included without exception.
Equivalently, with
$\beta:=\tfrac{a-1}{b}-\tfrac12+\tfrac1{bt}$,
\begin{equation}\label{eq:gf-ab}
\sum_{n\ge0}\ehr_{\X_n(a,b)}(t)\,\frac{z^n}{n!}
=\frac{1}{\sqrt{1-T(btz)}}\,\exp\!\left[\beta\,T(btz)-\frac{T(btz)^2}{4bt}\right].
\end{equation}
\end{theorem}

\begin{proof}
Abbreviate the weight appearing in \eqref{eq:graphsum} by
\[
w(G):=\bigl((a-1)t\bigr)^{\#\mathrm{loops}(G)}(bt)^{\#\mathrm{single}(G)}
\left(\frac{bt(bt+1)}{2}\right)^{\#\mathrm{pairs}(G)},
\]
and let $\mathcal G_{\mathrm C}(m)$ denote the set of connected graphs in $\mathcal G(m)$.
The statistics $\#\mathrm{loops}$, $\#\mathrm{single}$, and $\#\mathrm{pairs}$ are additive over connected components, so $w$ is multiplicative over connected components.
Setting $\ehr_{\X_0(a,b)}(t):=1$, the exponential formula \cite[Corollary~5.1.6]{StanleyEC2} applied to Proposition~\ref{prop:graphsum} gives
\begin{equation}\label{eq:expformula}
\sum_{n\ge0}\ehr_{\X_n(a,b)}(t)\,\frac{z^n}{n!}=e^{f(t,z)},
\qquad
f(t,z):=\sum_{m\ge1}\ \sum_{G\in\mathcal G_{\mathrm C}(m)}w(G)\,\frac{z^m}{m!}.
\end{equation}

A connected graph whose (unique) component contains at most one cycle is of exactly one of the following four kinds \cite[Section~4.3]{Behrend}, illustrated in Figure~\ref{fig:graphs}: a \emph{tree}; a \emph{looped tree} (a tree with a loop attached at one vertex); an \emph{enhanced tree} (a tree with one edge converted to an edge pair); or a \emph{quasitree} (a connected simple graph with exactly one cycle, of length at least three).
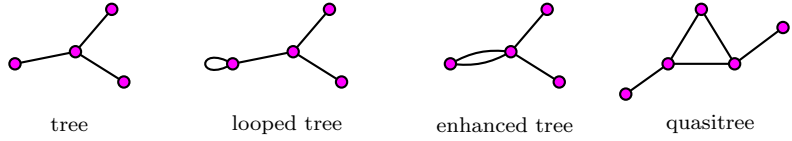
\begin{figure}[ht]
    \centering
    \begin{tikzpicture}[scale=0.8]
\begin{scope}
  \draw[edge] (0,0)--(1,0.2) (1,0.2)--(1.8,-0.3) (1,0.2)--(1.6,0.9);
  \foreach \p in {(0,0),(1,0.2),(1.8,-0.3),(1.6,0.9)} \node[vertex] at \p {};
  \node[font=\scriptsize] at (0.9,-1.0) {tree};
\end{scope}
\begin{scope}[xshift=3.6cm]
  \draw[edge] (0,0)--(1,0.2) (1,0.2)--(1.8,-0.3) (1,0.2)--(1.6,0.9);
  \draw[edge] (0,0) to[out=150,in=210,min distance=7mm] (0,0);
  \foreach \p in {(0,0),(1,0.2),(1.8,-0.3),(1.6,0.9)} \node[vertex] at \p {};
  \node[font=\scriptsize] at (0.9,-1.0) {looped tree};
\end{scope}
\begin{scope}[xshift=7.2cm]
  \draw[edge] (0,0) to[bend left=18] (1,0.2) (0,0) to[bend right=18] (1,0.2);
  \draw[edge] (1,0.2)--(1.8,-0.3) (1,0.2)--(1.6,0.9);
  \foreach \p in {(0,0),(1,0.2),(1.8,-0.3),(1.6,0.9)} \node[vertex] at \p {};
  \node[font=\scriptsize] at (0.9,-1.0) {enhanced tree};
\end{scope}
\begin{scope}[xshift=10.8cm]
  \draw[edge] (0,0)--(1.1,0)--(0.55,0.9)--cycle (1.1,0)--(1.9,0.6) (0,0)--(-0.7,-0.5);
  \foreach \p in {(0,0),(1.1,0),(0.55,0.9),(1.9,0.6),(-0.7,-0.5)} \node[vertex] at \p {};
  \node[font=\scriptsize] at (0.7,-1.0) {quasitree};
\end{scope}
\end{tikzpicture}
    \caption{The four kinds of connected graphs in $\mathcal G_{\mathrm C}(m)$.}
    \label{fig:graphs}
\end{figure}

On such a graph with $m$ vertices, the triple $(\#\mathrm{loops},\#\mathrm{single},\#\mathrm{pairs})$ equals $(0,m-1,0)$, $(1,m-1,0)$, $(0,m-2,1)$, and $(0,m,0)$, respectively.
So,
\[
w(G)=(bt)^{m-1},\qquad
(a-1)t\,(bt)^{m-1},\qquad
(bt)^{m-2}\,\frac{bt(bt+1)}{2},\qquad
(bt)^{m},
\]
that is, $w(G)=c\,(bt)^m$ with
\[
c=\frac{1}{bt},\qquad
\frac{a-1}{b},\qquad
\frac{bt+1}{2bt},\qquad
1,
\]
for trees, looped trees, enhanced trees, and quasitrees, respectively.
By Cayley's formula there are $m^{m-2}$ trees on $m$ labeled vertices, hence, $m^{m-1}$ looped trees and $(m-1)m^{m-2}$ enhanced trees; the exponential generating functions of the four families are expressed through the tree function \eqref{eq:treefn} as
\[
T(z)-\frac{T(z)^2}{2},\qquad
T(z),\qquad
\frac{T(z)^2}{2},\qquad
-\frac{T(z)}{2}-\frac{T(z)^2}{4}-\log\sqrt{1-T(z)},
\]
respectively; see \cite[Eqs.~(24)--(27)]{Behrend}, where the quasitree series is quoted from \cite[Eq.~(3.5)]{JKLP}.
Substituting into \eqref{eq:expformula} and writing $w:=T(btz)$, we obtain
\[
f(t,z)
=\frac{1}{bt}\Bigl(w-\frac{w^2}{2}\Bigr)
+\frac{a-1}{b}\,w
+\frac{bt+1}{2bt}\cdot\frac{w^2}{2}
-\frac{w}{2}-\frac{w^2}{4}-\log\sqrt{1-w}
=\beta\,w-\frac{w^2}{4bt}-\log\sqrt{1-w},
\]
since the coefficient of $w$ collects to $\frac1{bt}+\frac{a-1}{b}-\frac12=\beta$ and the coefficient of $w^2$ collects to $-\frac{1}{2bt}+\frac{bt+1}{4bt}-\frac14=-\frac{1}{4bt}$.
Exponentiating gives \eqref{eq:gf-ab}.

For \eqref{eq:ehrhart-ab}, write $g(w):=e^{\beta w-w^2/(4bt)}/\sqrt{1-w}$, so that \eqref{eq:gf-ab} reads
\[
\ehr_{\X_n(a,b)}(t)=n!\,(bt)^n\,[u^n]\,g(T(u)),
\]
where $[u^n]$ denotes coefficient extraction.
The tree function satisfies $[u^n]\,\varphi(T(u))=[u^n]\,\varphi(u)(1-u)e^{nu}$ for every formal power series $\varphi$ \cite[Eq.~(2.38)]{CorlessEtAl}, so that
\begin{equation}\label{eq:extracted}
\ehr_{\X_n(a,b)}(t)=n!\,(bt)^n\,[u^n]\ \sqrt{1-u}\ e^{(n+\beta)u-u^2/(4bt)}.
\end{equation}
Using the expansions
\[
[u^p]\sqrt{1-u}=-\frac{(2p-3)!!}{2^p\,p!},\qquad
[u^q]\,e^{(n+\beta)u}=\frac{(n+\beta)^q}{q!},\qquad
[u^{2i}]\,e^{-u^2/(4bt)}=\frac{(-1)^i}{i!\,(4bt)^i},
\]
with the double-factorial convention of the theorem statement, and parametrizing the coefficient of $u^n$ in the threefold product by $j:=p+2i$ (so $p=j-2i$ and $q=n-j$), \eqref{eq:extracted} becomes
\[
\ehr_{\X_n(a,b)}(t)
=\sum_{i=0}^{\lfloor n/2\rfloor}\sum_{j=2i}^{n}
(-1)^{i+1}\,\frac{n!}{(n-j)!\,(j-2i)!\,i!}\,
\frac{(2j-4i-3)!!}{2^{j-2i}\,4^{\,i}}\,
(bt)^{\,n-i}\,(n+\beta)^{n-j}.
\]
Since $2^{j-2i}4^i=2^j$ and $bt(n+\beta)=\tfrac12\bigl((2bn-b+2a-2)t+2\bigr)$, rewriting $(bt)^{n-i}(n+\beta)^{n-j}=(bt)^{j-i}\bigl(bt(n+\beta)\bigr)^{n-j}$ and inserting $i!/i!$ yields \eqref{eq:ehrhart-ab}.
\end{proof}

\begin{example}\label{ex:running-ab}
The running example $\X_3(3,2,2)=\X_3(3,2)$ has $(n,a,b)=(3,3,2)$.
Corollary~\ref{lem:minkowski} expresses its lifting as the Minkowski sum
\[
\overline{\X}_3(3,2)
=\sum_{i=1}^{3}\triangle_{\{i\}}
+2\sum_{i=1}^{3}\triangle_{\{i,\,4\}}
+2\!\!\sum_{1\le i<j\le 3}\!\!\triangle_{\{i,j,\,4\}},
\]
and the sum \eqref{eq:graphsum} of Proposition~\ref{prop:graphsum} runs over the $51$ graphs in $\mathcal G(3)$, each loop and each single edge carrying the weight $2t$ and each edge pair the weight $t(2t+1)$.
Theorem~\ref{thm:ehrhart-ab} evaluates to
\[
\ehr_{\X_3(3,2)}(t)=172t^3+84t^2+15t+1,
\]
in agreement with Example~\ref{ex:running-ehrhart}.
\end{example}

\begin{remark}\label{rem:behrend}
When $b=1$, the polytope $\X_n(a,1)$ is integrally equivalent to the partial permutahedron $\mathcal P(n,n+a-2)$ \cite[Proposition~3.16]{HanadaLentferVindas}, and Theorem~\ref{thm:ehrhart-ab} specializes to Behrend's Ehrhart formula for partial permutahedra \cite{Behrend}; likewise, the graph expansion of Proposition~\ref{prop:graphsum} reduces to his at $b=1$.
Theorem~\ref{thm:ehrhart-ab} may thus be viewed as the two-parameter extension of Behrend's formula.
\end{remark}

%%%%%%%%%%%%%%%%%%%%%%%%%%%%

\section{Magic positivity}\label{sec:magic}

In this section we apply Theorem~\ref{thm:ehrhart-ab} to establish a strong positivity property of the Ehrhart polynomials of the polytopes $\X_n(a,b)$.

\begin{definition}[{\cite[Definition~1.1]{LiuZhang}; see also \cite[Section~4.2]{FerroniHigashitani}}]\label{def:magic}
A lattice polytope $P$ of dimension $d$ is \emph{magic positive} if its Ehrhart polynomial can be expressed in the basis $\{t^i(t+1)^{d-i}\}_{i=0}^{d}$ with nonnegative coefficients, that is, if
$\ehr_P(t)=\sum_{i=0}^{d}\mu_i\,t^i(t+1)^{d-i}$ with $\mu_i\ge0$ for all $i$.
\end{definition}

Expanding each basis element shows that magic positivity implies Ehrhart positivity.
Moreover, by a result of Br\"and\'en \cite{Branden}, it implies that the $h^*$-polynomial is real-rooted; see also \cite{LiuZhang}.

Ferroni and Higashitani \cite[Problem~4.23]{FerroniHigashitani} include generalized parking-function polytopes among the families whose magic positivity is to be determined.
In this section we give a complete answer for the two-parameter family $\X_n(a,b)=\X_n(a,b,\dots,b)$ studied in \cite{HanadaLentferVindas}.
The specialization $b=1$ recovers, up to integral equivalence, the stable-range partial permutahedra considered by Liu and Zhang \cite{LiuZhang}: $\X_n(a,1)\cong\mathcal P(n,n+a-2)$ \cite[Proposition~3.16]{HanadaLentferVindas}.
Our theorem shows that no additional exceptions occur when the parameter $b$ is allowed to vary.

\begin{theorem}\label{thm:magic}
Let $a$, $b$, and $n$ be positive integers.
The polytope $\X_n(a,b)$ is magic positive if and only if $(n,a,b)\ne(2,1,1)$.
In the exceptional case, $\X_2(1,1)=PF_2$ and $\ehr_{PF_2}(t)=\tfrac{(t+1)(t+2)}{2}=(t+1)^2-\tfrac12\,t(t+1)$.
\end{theorem}

\begin{remark}[Relation to Ferroni--Higashitani's problem]\label{rem:FerroniHigashitani}
Theorem~\ref{thm:magic} gives a complete answer to \cite[Problem~4.23]{FerroniHigashitani} for the generalized parking-function polytopes $\X_n(a,b,\dots,b)$ studied by Hanada, Lentfer, and Vindas-Mel\'endez \cite{HanadaLentferVindas}: every polytope in this family is magic positive except for the classical parking-function polytope $\X_2(1,1)=PF_2$.
The theorem does not settle magic positivity for arbitrary $\X_n(\vb)$; that broader question is stated as Conjecture~\ref{conj:magicgeneral}.
\end{remark}

Our proof follows the strategy that Liu and Zhang developed for partial permutahedra \cite{LiuZhang}, with two modifications: the starting point is the coefficient-extraction formula \eqref{eq:extracted}, which plays the role of Behrend's formula \cite[Eq.~(30)]{Behrend}, and the auxiliary series acquires a deformation parameter $\tfrac1b$ whose effect on the coefficient estimates must be controlled.
The proof of Theorem~\ref{thm:magic} proceeds in three steps.
\begin{enumerate}
    \item First, we convert magic positivity into an explicit nonnegativity problem: applying the magic transform to \eqref{eq:extracted} shows that the coefficient of $t^i(t+1)^{n-i}$ in $\ehr_{\X_n(a,b)}(t)$ is a positive multiple of the quantity $B_{i,\,n-i}\bigl(n+\frac{a-2}{b}\bigr)$ defined in \eqref{eq:Bir} (Lemma~\ref{lem:magiccoeffs}).

    \item Second, we analyze the boundary values $B_{i,r}\bigl(i+r-\tfrac1b\bigr)$: Lagrange inversion expresses them as coefficients of the series $R(u)^rC(u)$ of \eqref{eq:RC} (Lemma~\ref{lem:boundary}), where $R$ is the series studied by Liu and Zhang while $C$ deforms with $\tfrac1b$, and the coefficient estimates of Lemmas~\ref{lem:R} and~\ref{lem:C} establish their nonnegativity (Proposition~\ref{prop:boundarypos}).

    \item Finally, since differentiation with respect to the parameter $x$ lowers the first index of $B_{i,r}(x)$ (Lemma~\ref{lem:lower}), an induction on $i$ propagates nonnegativity from the boundary to all $x\ge i+r-\tfrac1b$ (Proposition~\ref{prop:allx}); the exceptional pair $(i,r)=(1,1)$, which occurs only when $n=2$, accounts for the single exception $\X_2(1,1)$.
\end{enumerate}
Throughout this section we fix a positive integer $b$ and write $\delta:=\tfrac1b\in(0,1]$.

Following \cite[Section~2]{LiuZhang}, for a polynomial $f$ of degree at most $d$ define its \emph{degree-$d$ magic transform} by
\[
M_d f(y)\ :=\ (1-y)^d\,f\!\left(\frac{y}{1-y}\right).
\]
Then $f(t)=\sum_{i=0}^d\mu_i\,t^i(1+t)^{d-i}$ if and only if $M_df(y)=\sum_{i=0}^d\mu_i\,y^i$, so magic positivity of a $d$-dimensional lattice polytope is equivalent to coefficientwise nonnegativity of the degree-$d$ magic transform of its Ehrhart polynomial.

For integers $i,r\ge0$ and a real number $x$, define
\begin{equation}\label{eq:Bir}
B_{i,r}(x)\ :=\ [z^i]\,(1-z/4)^r\,\sqrt{1-z}\ \exp\!\left(\Bigl(x-\frac12\Bigr)z+\frac{z^2}{4b}\right),
\end{equation}
where the dependence on $b$ is suppressed from the notation.

\begin{lemma}\label{lem:magiccoeffs}
Let $n\ge2$ and write $M_n\ehr_{\X_n(a,b)}(y)=\sum_{i=0}^{n}\mu_i\,y^i$.
Then, for $0\le i\le n$,
\[
\mu_i\ =\ \frac{n!\ b^{\,i}}{(n-i)!}\ B_{i,\,n-i}\!\left(n+\frac{a-2}{b}\right).
\]
\end{lemma}

\begin{proof}
We compute as in \cite[Section~2]{LiuZhang}, starting from \eqref{eq:extracted}.
Substituting $t=\frac{y}{1-y}$, so that $bt=\frac{by}{1-y}$ and $\frac{1}{bt}=\frac{1}{by}-\frac1b$, and recalling $\beta=\frac{a-1}{b}-\frac12+\frac{1}{bt}$, the exponent in \eqref{eq:extracted} becomes
\[
(n+\beta)u-\frac{u^2}{4bt}
=\Bigl(n+\frac{a-1}{b}-\frac12\Bigr)u+\frac{u-u^2/4}{bt}
=\Bigl(x-\frac12\Bigr)u+\frac{u^2}{4b}+\frac{u-u^2/4}{by},
\]
where $x:=n+\frac{a-2}{b}$.
Hence,
\[
M_n\ehr_{\X_n(a,b)}(y)
=(1-y)^n\,n!\,\Bigl(\frac{by}{1-y}\Bigr)^{\!n}[u^n]\,\sqrt{1-u}\;e^{(x-1/2)u+u^2/(4b)}\,\exp\!\left(\frac{u-u^2/4}{by}\right).
\]
Expanding $\exp\bigl((u-u^2/4)/(by)\bigr)=\sum_{r\ge0}\frac{(u-u^2/4)^r}{r!\,(by)^r}$ and noting that $(u-u^2/4)^r=u^r(1-u/4)^r$ has order at least $r$ in $u$, only the terms with $0\le r\le n$ survive the extraction of $[u^n]$, and the coefficient of $y^{\,i}=y^{\,n-r}$ is
\[
\mu_i=\frac{n!\,b^{\,n}}{r!\,b^{\,r}}\,[u^{\,n}]\,u^r(1-u/4)^r\sqrt{1-u}\,e^{(x-1/2)u+u^2/(4b)}
=\frac{n!\,b^{\,i}}{r!}\,B_{i,r}(x),\qquad r=n-i.\qedhere
\]
\end{proof}

Since the prefactor $\frac{n!\,b^i}{(n-i)!}$ is positive, Theorem~\ref{thm:magic} reduces to the nonnegativity of the quantities $B_{i,r}\bigl(n+\frac{a-2}{b}\bigr)$ with $i+r=n$.
Note that $a\ge1$ gives $n+\frac{a-2}{b}\ge n-\frac1b=i+r-\delta$, so the values of $x$ that must be controlled fill the ray $x\ge i+r-\delta$.
The next two lemmas record the elementary properties of $B_{i,r}$.

\begin{lemma}\label{lem:lower}
For $i\ge1$ and $r\ge0$ we have $\frac{\partial}{\partial x}B_{i,r}(x)=B_{i-1,r}(x)$, and $B_{0,r}(x)=1$.
Moreover, $B_{1,r}(x)=x-1-\frac r4$; in particular $B_{1,0}(x)=x-1$ and $B_{1,1}(x)=x-\frac54$.
\end{lemma}

\begin{proof}
The parameter $x$ occurs only in the factor $e^{(x-1/2)z}$, and differentiating it with respect to $x$ multiplies the series by $z$, which shifts coefficient extraction from $[z^i]$ to $[z^{i-1}]$.
The constant term of the product in \eqref{eq:Bir} is $1$, and its linear coefficient is $-\frac r4-\frac12+\bigl(x-\frac12\bigr)=x-1-\frac r4$.
\end{proof}

We express the boundary values $B_{i,r}(i+r-\delta)$ through the tree function \eqref{eq:treefn}.
Define
\begin{equation}\label{eq:RC}
R(u):=\frac{T(u)\bigl(1-T(u)/4\bigr)}{u}=e^{T(u)}\Bigl(1-\frac{T(u)}{4}\Bigr),
\qquad
C(u):=\frac{\exp\bigl(-\bigl(\frac12+\delta\bigr)T(u)+\frac{\delta}{4}T(u)^2\bigr)}{\sqrt{1-T(u)}}.
\end{equation}
The series $R$ is the one appearing in \cite[Eq.~(5)]{LiuZhang}; the series $C$ specializes to Liu and Zhang's at $\delta=1$.

\begin{lemma}\label{lem:boundary}
For all $i,r\ge0$,
\[
B_{i,r}\bigl(i+r-\delta\bigr)\ =\ [u^i]\,R(u)^r\,C(u).
\]
\end{lemma}

\begin{proof}
By the coefficient identity $[u^k]\,\varphi(T(u))=[z^k]\,\varphi(z)(1-z)e^{kz}$ \cite[Eq.~(2.38)]{CorlessEtAl} (see also \cite[Lemma~3.1]{LiuZhang}), applied with $k=i$ and
\[
\varphi(z)=(1-z/4)^r(1-z)^{-1/2}\exp\!\left(\Bigl(r-\delta-\frac12\Bigr)z+\frac{\delta}{4}z^2\right),
\]
we obtain
\[
[u^i]\,\varphi(T(u))
=[z^i]\,(1-z/4)^r\sqrt{1-z}\,\exp\!\left(\Bigl(i+r-\delta-\frac12\Bigr)z+\frac{\delta}{4}z^2\right)
=B_{i,r}(i+r-\delta).
\]
It remains to identify $\varphi(T(u))$.
Since $T=ue^{T}$, we have $e^{rT}=(T/u)^r$, so
\[
\varphi(T)=(1-T/4)^r(1-T)^{-1/2}\,e^{rT}\,e^{-(\delta+\frac12)T+\frac{\delta}{4}T^2}
=\left(\frac{T(1-T/4)}{u}\right)^{\!r}C(u)=R(u)^rC(u).\qedhere
\]
\end{proof}

Since $R$ does not depend on $\delta$, the coefficient estimates of Liu and Zhang import verbatim.

\begin{lemma}[{\cite[Lemmas~3.3 and~3.4]{LiuZhang}}]\label{lem:R}
Write $R(u)=\sum_{k\ge0}\rho_ku^k$ and, for $r\ge2$, $R(u)^r=\sum_{k\ge0}q_k^{(r)}u^k$.
Then $\rho_0=1$ and $\rho_k=\frac{(k+2)(k+1)^{k-2}}{2\,k!}$ for $k\ge1$; in particular, $\rho_1=\frac34$.
Moreover, all $\rho_k$ are nonnegative and $\rho_k\ge\rho_{k-1}$ for $k\ge2$.
Moreover, for every $r\ge2$ the coefficients $q^{(r)}_k$ are nonnegative and satisfy $q^{(r)}_k\ge q^{(r)}_{k-1}$ for all $k\ge1$.
\end{lemma}

The series $C$, on the other hand, deforms with $\delta$, and the required coefficient estimate is the following.

\begin{lemma}\label{lem:C}
Let $b$ be a positive integer and $\delta=\frac1b$.
Then
\[
C(u)\ =\ 1-\delta u+P(u),
\]
where $P(u)$ is a formal power series of order at least two with nonnegative coefficients.
\end{lemma}

\begin{proof}
Define $G(w):=(1-w)^{-1/2}\exp\bigl(\bigl(\frac12-\delta\bigr)w+\frac{\delta}{4}w^2\bigr)$, so that $C(u)=e^{-T(u)}\,G(T(u))$.
Then
\[
\log G(w)=\Bigl(\frac12-\delta\Bigr)w+\frac{\delta}{4}w^2+\sum_{j\ge1}\frac{w^j}{2j}
=(1-\delta)\,w+\frac{1+\delta}{4}\,w^2+\sum_{j\ge3}\frac{w^j}{2j}.
\]
Since $0<\delta\le1$, all coefficients of $\log G$ are nonnegative; consequently, writing $G(w)=\sum_{q\ge0}g_qw^q$, all $g_q$ are nonnegative, with
\[
g_0=1,\qquad g_1=1-\delta,\qquad g_2=\frac{1+\delta}{4}+\frac{(1-\delta)^2}{2},
\qquad\text{and}\qquad g_q\ \ge\ \frac{1}{2q}\quad(q\ge3),
\]
the last bound because $G=\exp(\log G)$ dominates $1+\log G$ coefficientwise when $\log G$ has nonnegative coefficients.
Since $T=ue^T$ gives $e^{-T}=u/T$, we have $e^{-T}T^q=uT^{q-1}$ for $q\ge1$, and hence,
\begin{equation}\label{eq:Cdecomp}
C(u)=e^{-T(u)}+g_1u+\sum_{q\ge2}g_q\,uT(u)^{q-1}.
\end{equation}
We now examine the coefficients $c_k:=[u^k]C(u)$, using the evaluations
\begin{equation}\label{eq:treecoeffs}
[u^k]\,e^{-T(u)}=-\frac{(k-1)^{k-1}}{k!}\quad(k\ge1),
\qquad
[u^N]\,T(u)^s=\frac{s}{N}\cdot\frac{N^{N-s}}{(N-s)!}\quad(N\ge s\ge1)
\end{equation}
from \cite[Eqs.~(11) and~(12)]{LiuZhang}.
Clearly $c_0=1$, and $c_1=-1+g_1=-\delta$.
For $k\ge2$, \eqref{eq:Cdecomp} gives
\begin{equation}\label{eq:cklower}
c_k=-\frac{(k-1)^{k-1}}{k!}+\sum_{q\ge2}g_q\,[u^{k-1}]\,T(u)^{q-1},
\end{equation}
in which every summand of the series is nonnegative.
It remains to show $c_k\ge0$ for $k\ge2$.

For $k=2$, only $q=2$ contributes and
\[
c_2=-\frac12+g_2=\frac{(1-\delta)(1-2\delta)}{4}.
\]
Since $\delta=\frac1b$ with $b$ a positive integer, either $\delta=1$ or $\delta\le\frac12$; in both cases $c_2\ge0$.
For $k=3$, only $q=2,3$ contribute, and $[u^2]T=[u^2]T^2=1$, so, with $s:=1-\delta\in[0,1)$ and $g_3=\frac16+\frac{(1-\delta)(1+\delta)}{4}+\frac{(1-\delta)^3}{6}$,
\[
c_3=-\frac23+g_2+g_3=\frac{s}{4}+\frac{s^2}{4}+\frac{s^3}{6}\ \ge\ 0.
\]
For $k\ge4$ we retain the terms $q=2,3,4$ of \eqref{eq:cklower} and discard the rest.
On $[0,1]$ the quadratic $g_2=\frac34-\frac34\delta+\frac{\delta^2}{2}$ attains its minimum $\frac{15}{32}$ at $\delta=\frac34$, while $g_3\ge\frac16$ and $g_4\ge\frac18$.
Together with the evaluations \eqref{eq:treecoeffs},
\[
[u^{k-1}]T=\frac{(k-1)^{k-3}}{(k-2)!},\qquad
[u^{k-1}]T^2=\frac{2(k-1)^{k-4}}{(k-3)!},\qquad
[u^{k-1}]T^3=\frac{3(k-1)^{k-5}}{(k-4)!},
\]
this yields
\begin{align*}
c_k&\ \ge\ \frac{(k-1)^{k-5}}{k!}\left[-(k-1)^4+\frac{15}{32}\,k(k-1)^3+\frac13\,k(k-1)^2(k-2)+\frac38\,k(k-1)(k-2)(k-3)\right]\\
&\ =\ \frac{(k-1)^{k-5}}{96\;k!}\,\bigl(17k^4-95k^3+115k^2+59k-96\bigr).
\end{align*}
Finally, $17k^4-95k^3+115k^2+59k-96=(k-4)\bigl(17k^3-27k^2+7k+87\bigr)+252$, and for $k\ge4$ the cubic factor is positive, so $c_k>0$ for all $k\ge4$.
This proves $C(u)=1-\delta u+P(u)$ with $P$ of order at least two and nonnegative coefficients.
\end{proof}

\begin{proposition}\label{prop:boundarypos}
Let $b$ be a positive integer, $\delta=\frac1b$, and let $i,r\ge0$ satisfy $i+r\ge2$.
Then
\[
B_{i,r}(i+r-\delta)\ \ge\ 0,
\]
except for the single pair $(i,r)=(1,1)$, for which $B_{1,1}(2-\delta)=\frac34-\delta$; this value is negative if and only if $b=1$.
\end{proposition}

\begin{proof}
By Lemma~\ref{lem:boundary}, $B_{i,r}(i+r-\delta)=[u^i]R(u)^rC(u)$, and by Lemma~\ref{lem:C},
\begin{equation}\label{eq:threeterm}
[u^i]R(u)^rC(u)=q^{(r)}_i-\delta\,q^{(r)}_{i-1}+[u^i]R(u)^rP(u),
\end{equation}
with the conventions $q^{(0)}_k=[k=0]$, $q^{(1)}_k=\rho_k$, and $q^{(r)}_{-1}=0$.
The last term of \eqref{eq:threeterm} is nonnegative, because $R$ and $P$ have nonnegative coefficients.
If $r=0$, then $i\ge2$ and $[u^i]C(u)=c_i\ge0$ by Lemma~\ref{lem:C}.
If $r=1$ and $i\ge2$, then $\rho_i-\delta\rho_{i-1}\ge\rho_i-\rho_{i-1}\ge0$ by Lemma~\ref{lem:R}.
If $r=1$ and $i=1$, then, since $P$ has order at least two, $[u]RP=0$ and the value is exactly $\rho_1-\delta\rho_0=\frac34-\delta$.
If $r\ge2$, then $q^{(r)}_i\ge q^{(r)}_{i-1}\ge\delta\,q^{(r)}_{i-1}$ by Lemma~\ref{lem:R}.
\end{proof}

\begin{proposition}\label{prop:allx}
Let $b$ be a positive integer and $\delta=\frac1b$, and let $i,r\ge0$ satisfy $i+r\ge2$ and $(i,r)\ne(1,1)$.
Then
\[
B_{i,r}(x)\ \ge\ 0\qquad\text{for every }x\ge i+r-\delta.
\]
\end{proposition}

\begin{proof}
We induct on $i$, simultaneously for all $r\ge0$.
If $i=0$, then $B_{0,r}=1$ by Lemma~\ref{lem:lower}.
Let $i\ge1$, set $N:=i+r-\delta$, and suppose the claim holds for all smaller first indices.
By Lemma~\ref{lem:lower},
\[
B_{i,r}(x)=B_{i,r}(N)+\int_N^x B_{i-1,r}(s)\,ds,
\]
and $B_{i,r}(N)\ge0$ by Proposition~\ref{prop:boundarypos}, since $(i,r)\ne(1,1)$.
It remains to check that $B_{i-1,r}(s)\ge0$ for all $s\ge N$.
If $(i-1)+r\ge2$ and $(i-1,r)\ne(1,1)$, this follows from the induction hypothesis, because $N=\bigl((i-1)+r-\delta\bigr)+1$.
If $(i-1,r)=(1,1)$, then $(i,r)=(2,1)$ and $N=3-\delta\ge2$, so $B_{1,1}(s)=s-\frac54\ge0$ for $s\ge N$ by Lemma~\ref{lem:lower}.
Finally, if $(i-1)+r\le1$, then, since $i+r\ge2$, $i\ge1$, and $(i,r)\ne(1,1)$, the only possibility is $(i,r)=(2,0)$, with $N=2-\delta\ge1$, and $B_{1,0}(s)=s-1\ge0$ for $s\ge N$ by Lemma~\ref{lem:lower}.
\end{proof}

\begin{proof}[Proof of Theorem~\ref{thm:magic}]
First let $n=1$.
If $a=1$, then $\X_1(1)$ is the single point $(1)$, of dimension zero, and $\ehr(t)=1$ is magic positive.
If $a\ge2$, then $\X_1(a)=[1,a]$ has dimension one and $\ehr(t)=(a-1)t+1=(t+1)+(a-2)t$, which is magic positive since $a\ge2$.

Now let $n\ge2$, so that $\dim\X_n(a,b)=n$ \cite[Proposition~2.2]{BBCDDDLMMNV}.
By Lemma~\ref{lem:magiccoeffs}, the magic coefficients are $\mu_i=\frac{n!\,b^i}{r!}B_{i,r}(x)$ with $r=n-i$ and $x=n+\frac{a-2}{b}$; since $a\ge1$,
\[
x\ \ge\ n-\frac1b\ =\ (i+r)-\delta .
\]
For every index $i$ with $(i,r)\ne(1,1)$, Proposition~\ref{prop:allx} gives $\mu_i\ge0$.
The pair $(i,r)=(1,1)$ occurs only when $n=2$, in which case Lemma~\ref{lem:lower} gives
\[
\mu_1=2b\,B_{1,1}\Bigl(2+\frac{a-2}{b}\Bigr)=2b\left(\frac34+\frac{a-2}{b}\right)=\frac{3b}{2}+2(a-2).
\]
If $a\ge2$, this is positive; if $a=1$, it equals $\frac{3b-4}{2}$, which is nonnegative if and only if $b\ge2$.
Hence, all magic coefficients are nonnegative whenever $(n,a,b)\ne(2,1,1)$.
For $(n,a,b)=(2,1,1)$, we get $\mu_1=-\frac12$, so $\X_2(1,1)=PF_2$ is not magic positive; its Ehrhart polynomial is $\frac{(t+1)(t+2)}{2}=(t+1)^2-\frac12\,t(t+1)$.
\end{proof}

\begin{example}\label{ex:running-magic}
For the running example $(n,a,b)=(3,3,2)$, we have $\delta=\tfrac12$ and $x=n+\tfrac{a-2}{b}=\tfrac72$, and direct computation of the coefficients \eqref{eq:Bir} gives
\[
B_{0,3}\bigl(\tfrac72\bigr)=1,\qquad
B_{1,2}\bigl(\tfrac72\bigr)=2,\qquad
B_{2,1}\bigl(\tfrac72\bigr)=\tfrac{19}{8},\qquad
B_{3,0}\bigl(\tfrac72\bigr)=\tfrac{17}{8},
\]
the second of which also follows from Lemma~\ref{lem:lower}.
Lemma~\ref{lem:magiccoeffs} then yields $(\mu_0,\mu_1,\mu_2,\mu_3)=(1,12,57,102)$; indeed,
\[
\ehr_{\X_3(3,2)}(t)=(t+1)^3+12\,t(t+1)^2+57\,t^2(t+1)+102\,t^3,
\]
exhibiting the magic positivity of $\X_3(3,2,2)$ concretely.

The value $B_{2,1}\bigl(\tfrac72\bigr)$ also illustrates the mechanism of the proof.
At the boundary point $2+1-\delta=\tfrac52$, Lemmas~\ref{lem:boundary}, \ref{lem:R}, and~\ref{lem:C} give
\[
B_{2,1}\bigl(\tfrac52\bigr)=[u^2]\,R(u)\,C(u)=\rho_2c_0+\rho_1c_1+\rho_0c_2
=1-\tfrac34\cdot\tfrac12+0=\tfrac58,
\]
using $(\rho_0,\rho_1,\rho_2)=\bigl(1,\tfrac34,1\bigr)$ and $(c_0,c_1,c_2)=\bigl(1,-\tfrac12,0\bigr)$; note that $c_2=\frac{(1-\delta)(1-2\delta)}{4}$ vanishes at $\delta=\tfrac12$.
Lemma~\ref{lem:lower} then propagates positivity to $x=\tfrac72$, as in the proof of Proposition~\ref{prop:allx}:
\[
B_{2,1}\bigl(\tfrac72\bigr)=B_{2,1}\bigl(\tfrac52\bigr)+\int_{5/2}^{7/2}B_{1,1}(s)\,ds
=\tfrac58+\int_{5/2}^{7/2}\Bigl(s-\tfrac54\Bigr)ds=\tfrac58+\tfrac74=\tfrac{19}{8}.
\]
Finally, the $h^*$-polynomial of $\X_3(3,2)$ is $h^*(z)=102z^3+661z^2+268z+1$, whose three roots (approximately $-6.046$, $-0.431$, and $-0.004$) are all real, as guaranteed by Corollary~\ref{cor:realrooted} below.
\end{example}

\begin{corollary}\label{cor:realrooted}
For all positive integers $a$, $b$, and $n$, the $h^*$-polynomial of $\X_n(a,b)$ is real-rooted.
\end{corollary}

\begin{proof}
For $(n,a,b)\ne(2,1,1)$ this follows from Theorem~\ref{thm:magic} and Br\"and\'en's result \cite{Branden} that magic positivity implies real-rootedness of the $h^*$-polynomial (see \cite[Section~1]{LiuZhang}).
The exceptional polytope $\X_2(1,1)=PF_2$ is a unimodular triangle, so $h^*_{PF_2}(x)=1$, which is vacuously real-rooted.
\end{proof}

\begin{remark}\label{rem:LZcompare}
Via the integral equivalence $\X_n(a,1)\cong\mathcal P(n,n+a-2)$ of \cite[Proposition~3.16]{HanadaLentferVindas}, the case $b=1$ of Theorem~\ref{thm:magic} is precisely the theorem of Liu and Zhang \cite[Theorem~1.3]{LiuZhang} for partial permutahedra $\mathcal P(m,p)$ in the stable range $p\ge m-1$, with the exceptional pair $(m,p)=(2,1)$ corresponding to $(n,a,b)=(2,1,1)$.
For $b\ge2$ the polytopes $\X_n(a,b)$ are not partial permutahedra, and Theorem~\ref{thm:magic} shows that no new exceptions arise anywhere in the two-parameter family.

A complementary large-parameter result was obtained by Avila, Ferroni, and Morales \cite{AvilaFerroniMorales}, who proved magic positivity for type-$\mathcal Y$ generalized permutahedra whose simplex coefficients are sufficiently large relative to the dimension.
Since $\overline{\X}_n(a,b)$ is such a polytope, with coefficients $a-1$ and $b$ (Corollary~\ref{lem:minkowski}), their result covers the region of parameters meeting those bounds; Theorem~\ref{thm:magic} is complementary in that it reaches all positive $a$ and $b$, including the small values, and the unique exception, that lie beyond such large-coefficient criteria.
\end{remark}

%%%%%%%%%%%%%%%%%%%%%%%%%%%%

\section{Future directions}\label{sec:future}

We conclude with some directions for further study.

Theorem~\ref{thm:magic} settles the magic-positivity question of Ferroni and Higashitani \cite[Problem~4.23]{FerroniHigashitani} for the two-parameter family $\X_n(a,b,\dots,b)$; the corresponding problem for arbitrary parameter vectors $\vb\in\Z_{>0}^n$ remains open.
Every $\vb\in\Z_{>0}^2$ is of the form $(a,b)$, so Theorem~\ref{thm:magic} also settles magic positivity for all $\vb$-parking-function polytopes of dimension at most two.
Beyond the family $\vb=(a,b,\dots,b)$, we have used Theorem~\ref{thm:general} to compute the Ehrhart polynomial of $\X_n(\vb)$ and its magic transform for all $\vb\in\{1,2,3,4\}^3$ and for a variety of parameter vectors of length four; in every case the magic coefficients are nonnegative.
For instance, for $\vb=(2,3,1)$ (Example~\ref{ex:running-ehrhart}) the magic transform of the Ehrhart polynomial is $1+\tfrac{59}{6}y+\tfrac{115}{3}y^2+54y^3$.
This evidence suggests the following.

\begin{conjecture}\label{conj:magicgeneral}
For every $n\ge3$ and every $\vb\in\Z_{>0}^n$, the polytope $\X_n(\vb)$ is magic positive.
Equivalently, the classical parking-function polytope $PF_2=\X_2(1,1)$ is the only $\vb$-parking-function polytope that is not magic positive.
\end{conjecture}

By \cite{Branden}, Conjecture~\ref{conj:magicgeneral} would imply that the $h^*$-polynomial of every $\vb$-parking-function polytope is real-rooted, and in particular that its coefficients are log-concave and unimodal.
A natural approach would be to extend the analysis of Section~\ref{sec:magic}; the obstacle is that for general $\vb$ the generating function of Theorem~\ref{thm:general} no longer collapses to graphs with at most one cycle per component, and no analogue of the coefficient-extraction formula \eqref{eq:extracted} is currently available.
It would also be natural to investigate magic positivity across the larger family $PF(\mathbf u)$ of Liu and Thawinrak \cite{LiuThawinrak}, where ties among the entries of $\mathbf u$ are allowed.

A second direction concerns the lattice-point polynomial $\Lp_n$ of Proposition~\ref{prop:polynomiality}.
Theorem~\ref{thm:general} expresses $\Lp_n$ as a sum over draconian sequences, but it would be interesting to have a more structured expression, for example a determinantal one.
The enumeration of $\vb$-parking functions themselves is governed by Gon\v{c}arov polynomials \cite{KungYan}; since the lattice points of $\X_n(\vb)$ contain, in general properly, the $\vb$-parking functions (see \cite[Section~2]{BBCDDDLMMNV}), one may ask whether $\Lp_n$ admits a comparable theory.

Finally, the slice description of Theorem~\ref{thm:slice} invites refinement: the layer polytopes $\X_{n-1}(\vb')$ interpolate between adjacent coordinates of $\vb$ as the height increases, and it would be interesting to understand how finer invariants, such as the $h^*$-polynomial or the combinatorial type in the sense of \cite[Corollary~3.14]{BBCDDDLMMNV}, evolve along the layers.

%--------------------
\section*{Tool and Computational Resource Disclosure}
During the preparation of this work, the authors used Claude Opus 4.8 to brainstorm the proof strategy of Proposition~\ref{prop:polynomiality}. These discussions led the authors to consider the use of Faulhaber's formula and to identify the references cited in the proof. 
After using these tools, the authors reviewed, edited, and verified all resulting content as needed and take full responsibility for the content of the article. 
The mathematical arguments, proofs, and exposition were written, checked for correctness, and approved by the authors.
The authors also used SageMath to perform computations and verify examples appearing in the article.

%--------------------Acknowledgements--------------------%
\section*{Acknowledgments} 
The authors are grateful to the Mathematics Department at Harvey Mudd College for providing a wonderful environment to produce their work.
The authors also thank Rishi Mathur for her assistance with SageMath computations. 
CH, AL, and ARVM are partially supported by the NSF under Award DMS-2532321.

%--------------------Bibliography--------------------%

\bibliographystyle{amsplain}
\bibliography{bibliography}

\end{document}